\documentclass[11pt,twoside,reqno]{amsart}

\usepackage{microtype}
\usepackage[OT1]{fontenc}
\usepackage{type1cm}
\usepackage{amssymb}
\usepackage{mathtools}
\usepackage{esint}
\usepackage[left=2.3cm,top=3cm,right=2.3cm]{geometry}

\usepackage[active]{srcltx}
\usepackage{verbatim}

\usepackage[pdftex]{graphicx}
\usepackage{subfigure}

\numberwithin{equation}{section}

\theoremstyle{plain}
\newtheorem{theorem}{Theorem}[section]
\newtheorem*{informal}{Informal version of Theorem \ref{main}}
\newtheorem{corollary}[theorem]{Corollary}
\newtheorem{proposition}[theorem]{Proposition}
\newtheorem{lemma}[theorem]{Lemma}

\theoremstyle{remark}
\newtheorem{remark}[theorem]{Remark}

\newtheorem{example}[theorem]{Example}

\theoremstyle{definition}

\newcommand{\R}{\mathbb{R}}

\newcommand{\N}{\mathbb{N}}

\newcommand{\iii}{\mathtt{i}}
\newcommand{\jjj}{\mathtt{j}}

\newcommand{\eps}{\varepsilon}

\newcommand{\roo}{\varrho}

\renewcommand{\epsilon}{\varepsilon}
\renewcommand{\aa}{\mathtt{a}}
\newcommand{\uu}{\mathtt{u}}

\newcommand{\la}{\underline{\alpha}}
\newcommand{\ua}{\overline{\alpha}}

\newcommand{\OO}{\mathcal{O}}
\newcommand{\ZZ}{\mathcal{Z}}
\newcommand{\IN}{I^{\N}}
\newcommand{\ii}{\mathtt{i}}

\newcommand{\jj}{\mathtt{j}}

\newcommand{\iin}[1]{\ii|_{#1}}
\newcommand{\jjn}[1]{\jj|_{#1}}

\DeclareMathOperator{\dist}{dist}
\DeclareMathOperator{\diam}{diam}
\DeclareMathOperator{\diag}{diag}
\DeclareMathOperator{\proj}{proj}

\begin{document}

\title{Self-affine sets with fibered tangents}

\author{Antti K\"aenm\"aki}
\address{University of Jyvaskyla \\
         Department of Mathematics and Statistics \\
         P.O.\ Box 35 (MaD) \\
         FI-40014 University of Jyv\"askyl\"a \\
         Finland}
\email{antti.kaenmaki@jyu.fi}

\author{Henna Koivusalo}
\address{Department of Mathematics \\
         University of York \\
         Heslington \\
         York YO10 5DD \\
         UK}
\email{henna.koivusalo@york.ac.uk}

\author{Eino Rossi}
\address{University of Jyvaskyla \\
         Department of Mathematics and Statistics \\
         P.O.\ Box 35 (MaD) \\
         FI-40014 University of Jyv\"askyl\"a \\
         Finland}
\email{eino.rossi@jyu.fi}

\thanks{ER acknowledges the support of the Vilho, Yrj{\"o}, and Kalle V{\"a}is{\"a}l{\"a} 
foundation. Research of HK was supported by the EPSRC}
\date{\today}

\begin{abstract}
  We study tangent sets of strictly self-affine sets in the plane. If a set in this class satisfies the strong separation condition and projects to a line segment for sufficiently many directions, then for each generic point there exists a rotation $\OO$ such that all tangent sets at that point are either of the form $\OO((\R \times C) \cap B(0,1))$, where $C$ is a closed porous set, or of the form $\OO((\ell \times \{ 0 \}) \cap B(0,1))$, where $\ell$ is an interval.
\end{abstract}

\maketitle

\section{Introduction}\label{sec:introduction}

Taking tangents is a standard tool in analysis. Tangents are usually more regular than the original 
object and often they capture its local structure. Further, understanding how tangents behave at 
many points gives information about the global structure as well. For example, tangents of a 
differentiable function are affine maps, and they capture the full behavior of the function. 
Similarly, tangents of measures and sets are useful in understanding the fine structure of the 
objects under study, as well as their global properties. 

Tangent measures were introduced by Preiss \cite{Preiss1987} and they were a crucial ingredient in 
connecting densities to rectifiability. It should be noted that for general measures and sets 
tangents can be almost anything; see \cite{Buczolich2003, BuczolichRati2006, ChenRossi2015, 
Oneil1994, ONeil1995, Sahlsten2014}. However, under strong enough regularity assumptions, or when 
studying the tangents statistically, the tangent structure can describe the original object well. 
For example, recently Fraser \cite{Fraser2014} and Fraser, Henderson, Olson, and Robinson 
\cite{FraserHendersonOlsonRobinson2015} have used tangent sets to study dimensional properties of 
various fractal sets. This approach goes back to Mackay and Tyson \cite{MackayTyson2010} and Mackay 
\cite{Mackay2011}. Furthermore, the process of taking blow-ups of a measure or a set around a point 
induces a natural dynamical system. This makes it possible to apply ergodic-theoretical methods to 
understand the statistical behavior of tangents. The 
general theory related to this was initiated by Furstenberg \cite{Furstenberg1970, 
Furstenberg2008}. 
It was greatly developed by Hochman \cite{Hochman2010} and recently enhanced by K\"aenm\"aki, 
Sahlsten, and Shmerkin \cite{KaenmakiSahlstenShmerkin2013}. This ``zooming in'' dynamics has been 
considered for specific sets and measures arising from dynamics; see e.g.\ \cite{BedfordFisher1996, 
BedfordFisher1997, BedfordFisherUrbanski2002, FergusonFraserSahlsten2015, FraserPollicott2015, 
Gavish2011, Zahle1988}. The theory has found 
applications in arithmetics and geometric measure theory; see \cite{HochmanShmerkin2013, 
KaenmakiSahlstenShmerkin2014}.
Very recently, Kempton \cite{Kempton2015} studied the scenery flow of Bernoulli measures on 
self-affine sets associated to strictly positive matrices under the condition that projections of 
the self-affine measure
in typical directions are absolutely continuous.

The main purpose of this article is to investigate tangent sets of a self-affine set. A tangent set 
is a limit in the Hausdorff metric obtained from successive magnifications of the original set 
around a given point. These kind of local structures of fractal sets are of course interesting in 
their own right, but on top of that, our effort is motivated by the presumption that understanding 
them could provide new methods in the study of self-affine sets.

Tangent measures of self-similar sets have been given a satisfactory description by Bandt 
\cite{Bandt2001}. He studied tangent measures (and their distribution) of self-similar sets 
satisfying the open set condition, showing that almost every point has the same collection of 
tangent sets. Intuitively it is plausible that all tangent sets are homothetic copies of the 
original self-similar set. For a self-affine set, the expected tangent behaviour is completely 
different, and not yet very well understood. Under iteration by affine maps, balls are often mapped 
to narrower and narrower ellipses. Thus it is intuitive that, when zooming into a small ball, the 
magnification will contain narrow fibers in different directions. The limit object should hence be 
contained in a set consisting of long line segments, even if the original self-affine set is 
totally 
disconnected. Bandt and K\"aenm\"aki \cite{BandtKaenmaki2013} confirmed that this intuition is 
correct at least in the special case where all the mappings 
contract 
more in the vertical direction than in the horizontal direction. According to them, under a 
projection condition, tangent sets at generic points are product sets of a line and a perfect 
nowhere dense set. 

In the present article, we generalize the result of Bandt and K\"aenm\"aki to a general class of 
self-affine sets satisfying the strong separation condition and a projection condition. We emphasize 
that we do not need to assume the matrices in the affine iterated function system to be strictly 
positive. Therefore, in this sense, our setting is also more general than that of Kempton's. Some 
questions still remain open. It would be interesting to know when exactly the tangents of 
self-affine sets admit this kind of fibered tangent structure.

We begin the description of the setup by defining self-affine sets. Fix $m \in \N \setminus \{ 1 
\}$ 
and for each $i \in \{ 1,\ldots, m \}$ let $f_i \colon \R^2 \to \R^2$ be a contractive map with
\[
  f_i(x)=A_i x + b_i,
\]
where $A_i \in \R^{2 \times 2}$ is an invertible $2 \times 2$ matrix having operator norm strictly 
less than one, and $b_i\in \R^2$. The collection $\{f_1,\dots, f_m\}$ of affine mappings is an {\bf 
affine iterated function system}, and it gives rise to a {\bf self-affine set} $E$, which is the 
unique compact non-empty set satisfying 
\[
  E=\bigcup_{i=1}^m f_i(E).
\]
If the matrices $A_i$ are diagonal and the sets $f_i(E)$ are sufficiently separated, then the set 
$E$ is referred to as a {\bf self-affine carpet}. For example, if the {\bf strong separation 
condition} is satisfied, that is, $f_i(E) \cap f_j(E) = \emptyset$ for $i \ne j$, then there is a 
one-to-one correspondence between a point $x \in E$ and its address: the {\bf canonical projection} 
$\pi \colon \{ 1,\ldots,m \}^\N \to E$ defined by
\begin{equation} \label{eq:projection}
  \pi(i_1,i_2,\ldots) = \sum_{n=1}^\infty A_{i_1} \cdots A_{i_{n-1}} b_{i_n}
\end{equation}
is bijective.

It is clear that general results concerning tangent sets cannot be obtained for all points. 
Therefore we restrict our analysis to points which are generic with respect to Bernoulli measures. 
This is a natural class of measures to consider, since often the measure with maximal dimension is 
Bernoulli. It should be noted that a {\bf Bernoulli measure} $\nu_p$, defined to be the product of 
a 
given probability vector $p=(p_1,\ldots,p_m)$, is not a measure on $E$ but on $\{ 1,\ldots,m 
\}^\N$. 
Thus, precisely speaking, we consider generic points on $E$ with respect to the pushforward measure 
$\pi\nu_p$. But when the canonical projection is bijective we interpret $\nu_p$ as a measure on $E$.

Here we shall only state an informal version of our main theorem, as some of the assumptions are 
too 
technical to be presented in the introduction. A precise formulation of the theorem can be found in 
Theorem \ref{main}. We refer the reader to Remarks \ref{re:ontheassumptions2} and 
\ref{re:ontheassumptions} for sufficient and checkable conditions for the assumptions. Definitions 
of a tangent set and porosity are given in \S \ref{sec:zooming}. Lyapunov exponents will be defined 
in 
\S\ref{sec:orientation}; notice that their values depend on the Bernoulli measure. The closed unit 
ball is denoted by $B(0,1)$. 

\begin{informal}
  If a self-affine set $E$
  \begin{itemize}
  \item[(1)] satisfies the strong separation condition,
  \item[(2)] projects to a line segment for sufficiently many directions, 
  \item[(3)] has two distinct Lyapunov exponents,
  \end{itemize}
  then for $\nu_p$-almost every $x\in E$ there exists a rotation $\OO$ such that the tangent sets 
at 
$x$ are either of the form $\OO((\R\times C) \cap B(0,1))$, where $C$ is a closed porous set, or of 
the form $\OO((\ell\times \{0\}) \cap B(0,1))$, where $\ell$ is an interval containing at least one 
of the intervals $[-1,0]$ and $[0,1]$.
\end{informal}

It is reasonable to assume the strong separation condition since the geometry of the limit set for 
an overlapping iterated function system can differ hugely from the non-overlapping case. Further, 
the projection condition \eqref{as:projection} is necessary for the claim in this form to hold; see 
Example \ref{ex:noprojection}. The assumption on two distinct Lyapunov exponents guarantees that 
the 
system is ``strictly affine" -- this kind of assumption does not seem too restrictive since the 
expected behavior in the self-similar case is so different to the self-affine case. There could of 
course be a general theorem covering a wider class of iterated function systems, but even to guess 
for the statement of such a theorem seems very difficult. 

Let us now compare the assumptions of this result to the setting of Bandt and K\"aenm\"aki 
\cite{BandtKaenmaki2013}. They considered self-affine carpets on $[0,1]^2$ for which all the maps 
of 
iterated function systems contract more vertically than horizontally. They required the rectangular 
strong separation condition and also assumed that the projection of the carpet onto the $x$-axis is 
the whole interval $[0,1]$. Our assumptions (1)--(3) resemble their assumptions, but we can handle 
more general self-affine sets. In fact, our main result, Theorem \ref{main}, contains the 
self-affine carpets as a special case, regardless of whether there is a geometrically dominant 
contraction direction or not. For a discussion on this and other examples, see \S 
\ref{sec:discussion}. To finish with a concrete statement, we formulate a corollary which strictly 
generalizes the main result of \cite{BandtKaenmaki2013}. The proof is given in \S 
\ref{sec:discussion}.

\begin{corollary} \label{cor:carpet}
  Suppose that $\{ f_1,\ldots,f_m \}$ is an affine iterated function system defined on $[0,1]^2$ 
such that the linear part of $f_i$ is the diagonal matrix $\diag(h_i,v_i)$ with $0<h_i,v_i<1$ for 
all $i \in \{ 1,\ldots,m \}$ and
  \begin{enumerate}
    \item $f_i([0,1]^2) \cap f_j([0,1]^2) = \emptyset$ for $i \ne j$,
    \item \label{carpet_line} $\#\{ i \in \{ 1,\ldots,m \} : \{ c \} \times [0,1] \cap f_i([0,1]^2) 
\ne \emptyset \} \ge 2$ for all $c \in [0,1]$,
    \item \label{cor:carpet lyapunov} $-\sum_{i=1}^m \nu_p(f_i([0,1]^2))\log h_i < -\sum_{i=1}^m 
\nu_p(f_i([0,1]^2))\log v_i$.
  \end{enumerate}
  Then the tangent sets at $\nu_p$-almost every point of the associated self-affine carpet $E$ are 
of the form $(\R \times C) \cap B(0,1)$, where $C$ is a perfect porous set.
\end{corollary}

\section{Preliminaries}

In this section and throughout the article, the affine iterated function system $\{f_1, \dots, 
f_m\}$ and hence the self-affine set $E$ remain fixed. Further, $\nu_p$ is the Bernoulli measure 
corresponding to a probability vector $p=(p_1,\ldots,p_m)$. We let $\overline p = \max_{i \in \{ 
1,\ldots,m \} } p_i$ and $\underline p = \min_{i \in \{ 1,\ldots,m \} } p_i$.

\subsection{Symbolic space} We begin by presenting the symbolic representation of points in $E$. 
Let 
$I=\{1,\dots, m\}$ be the collection of letters. Define $\IN$ to be the collection of infinite 
words 
$\{1,\dots,m\}^\N$ and $I^k$ to be the set of finite words $\{1,\dots,m\}^k$ for all $k$. The empty 
word is denoted by $\varnothing$. Let $I^*=\bigcup_{k=0}^\infty I^k$ and denote by bold letters
$\ii,\jj, \aa, \uu$ and so on the finite or infinite words. For a finite word $\ii$, define the 
{\bf 
cylinder}
\[
  [\ii]=\{\jj\in\IN : \jjn{|\ii|}=\ii\},
\]
where $\jj|_k$ are the first $k$ letters of $\jj$ and $|\ii|$ is the length of the word $\ii$. We 
say that finite words $\ii$ and $\jj$ are {\bf incomparable}, $\ii\perp\jj$, if $[\iii] \cap [\jjj] 
= \emptyset$. Denote by $\ii(j)$ the $j$-th letter in $\ii$. For any finite word $\ii$, write 
$A_\ii=A_{\ii(1)} \cdots A_{\ii(|\ii|)}$ and $f_\ii=f_{\ii(1)}\circ\cdots\circ f_{\ii(|\ii|)}$. We 
refer to the images $f_\ii(E)$ as level $|\ii|$ {\bf construction cylinders} and denote them by 
$E_\ii$. Notice that $E_\ii = \pi([\ii])$, where $\pi$ is the canonical projection defined in 
\eqref{eq:projection}.

\subsection{Orientation of construction cylinders} \label{sec:orientation}
In this subsection, we study the orientation of the construction cylinders $E_\ii$. In particular, 
in Lemma \ref{Oselemma} we prove that they converge almost everywhere. This is a version of the 
Oseledets' theorem; see \cite[Theorem 10.2]{Walters1982} or \cite{Ruelle1979}. For the convenience 
of the reader, and since the statement does not seem to appear in the right form in the literature, 
we present a full proof.

Let us now make the above more precise. For each $\ii\in I^*$ and $k\in\{1,2\}$, let 
$\eta_k(\ii)\in 
S^1=\{x\in \R^2 : |x|=1\}$ be the eigenvectors of $A_{\ii}^T A_{\ii}$ with eigenvalues 
$\alpha_k^2(\ii)$ ordered so that $\alpha_1(\ii)\geq\alpha_2(\ii)$. We call $\eta_k(\ii)$ the {\bf 
singular vectors}, or {\bf singular directions}, of $A_{\ii}$ and $\alpha_k(\ii)$ the {\bf singular 
values} of $A_{\ii}$. From basic linear algebra it follows that
\[
 \alpha_1(\ii)= \| A_{\ii}\eta_1(\ii) \| =\max_{v\in S^1} \|A_\ii v\|
 \qquad\text{and}\qquad
 \alpha_2(\ii)= \| A_{\ii}\eta_2(\ii) \| =\min_{v\in S^1}\|A_\ii v\|.
\]
Furthermore, if $\alpha_1(\ii)>\alpha_2(\ii)$ we see that $\eta_1(\ii)$ is orthogonal to 
$\eta_2(\ii)$ and $A_{\ii}\eta_1(\ii)$ is orthogonal to $A_{\ii}\eta_2(\ii)$. The numbers 
$\alpha_1(\ii)$ and $\alpha_2(\ii)$ are the lengths of the principal semiaxes of the image of the 
unit ball $f_\ii(B(0,1))$. We write $\la=\min_{i \in \{ 1,\ldots,m \}} \alpha_2(i)$ and 
$\ua=\max_{i 
\in \{ 1,\ldots,m \}} \alpha_1(i)$. Since the matrices $A_i$ are invertible contractions we have 
$0<\la\leq \ua<1$.

Note that when $\alpha_1(\ii)=\alpha_2(\ii)$, the eigenspace of $A_\ii^T A_\ii$ is the whole 
$\R^2$. 
In this case, we can choose $\eta_1(\ii)$ and $\eta_2(\ii)$ to be any orthogonal vectors. Also when 
$\alpha_1(\ii) > \alpha_2(\ii)$, the sign of $\eta_k(\ii)$ can be chosen freely. This will not make 
a significant difference in what follows. For each $k \in \{ 1,2 \}$ and all finite words $\ii$, let
 \[
  \vartheta_k(\ii)=\frac{A_{\ii}\eta_k(\ii)}{\| A_{\ii}\eta_k(\ii) \|}. 
 \]
The pair $(\vartheta_1(\ii), \vartheta_2(\ii))$ gives the directions of the principal semiaxes of 
the ellipse $f_\ii(B(0,1))$, and hence the ``orientation'' of the construction cylinders $E_\ii$. 

For each $k \in \{ 1,2 \}$ and $\ii \in \IN$, we define a sequence by setting
\begin{equation}\label{eq:sequences}
\lambda_k(\iin{n})=-\frac{1}{n}\log \alpha_k(\iin{n}).
\end{equation}
Due to the subadditive ergodic theorem (see \cite[Theorem 10.1]{Walters1982} or \cite{Ruelle1979}), 
the {\bf Lyapunov exponents} $\lambda_k(\ii)=\lim_{n\to\infty}\lambda_k(\iin{n})$ exist for 
$\nu_p$-almost all $\ii\in \IN$ and for all $k \in \{ 1,2 \}$. Moreover, since Bernoulli measures 
are ergodic (for example, see \cite[Theorem 3.7]{KaenmakiReeve2014}) there exists a constant 
$\lambda_k$ so that $\lambda_k(\ii)=\lambda_k$ for $\nu_p$-almost all $\iii \in I^\N$ and for all 
$k 
\in \{ 1,2 \}$. Note that we always have $\lambda_2\geq\lambda_1$.

\begin{lemma} \label{Oselemma}
  If $\lambda_2>\lambda_1$, then for each $k \in \{ 1,2 \}$ and for $\nu_p$-almost every $\ii\in 
\IN$ there exists $\overline{\vartheta}_k(\ii)\in S^1$ so that 
$\vartheta_k(\iin{n})\to\overline{\vartheta}_k(\ii)$ as $n\to\infty$.
\end{lemma}

\begin{proof}
 First we notice that $\vartheta_k(\iin{n})$ are the singular directions of $A_{\iin{n}}^T$. This 
is 
true, since
 \begin{equation*}
  A_{\iin{n}} A_{\iin{n}}^T \vartheta_k(\iin{n})
 = A_{\iin{n}} \frac{ A_{\iin{n}}^T A_{\iin{n}}\eta_k(\iin{n})}{ \|A_{\iin{n}}\eta_k(\iin{n})\| }
  = \alpha_k(\iin{n})^2 \frac{ A_{\iin{n}}\eta_k(\iin{n}) }{ \|A_{\iin{n}}\eta_k(\iin{n})\| 
}=\alpha_k(\iin{n})^2\vartheta_k(\iin{n}).
 \end{equation*}
 By the orthogonality of $\vartheta_1(\iin{n})$ and $\vartheta_2(\iin{n})$, it suffices to show the 
convergence in the case $k=2$. For each $n\in\N$, let $\theta_n$ denote the angle between 
$\vartheta_2(\iin{n})$ and $\vartheta_2(\iin{n+1})$. We prove that $(\theta_n)$ is a Cauchy 
sequence. This yields the existence of the limit. Since $\vartheta_1(\iin{n+1})$ and 
$\vartheta_2(\iin{n+1})$ are orthogonal for all $n$, we can write
 \[
  \vartheta_2(\iin{n})= \sin(\theta_n) \vartheta_1(\iin{n+1})+ \cos(\theta_n) 
\vartheta_2(\iin{n+1}).
 \]
 Note that, by changing the sign if necessary, we can always choose $\vartheta_2(\iin{n})$ and 
$\vartheta_2(\iin{n+1})$ so that $-\pi/2\leq \theta_n \leq \pi/2$. Now we have
 \begin{align*}
  \|A_{\iin{n+1}}^T\vartheta_2(\iin{n})\|
  &= \|A_{\iin{n+1}}^T \sin(\theta_n) \vartheta_1(\iin{n+1})+ A_{\iin{n+1}}^T \cos(\theta_n) 
\vartheta_2(\iin{n+1})\| \\
  &\geq \|A_{\iin{n+1}}^T \sin(\theta_n) \vartheta_1(\iin{n+1})\| = |\sin(\theta_n)| 
\alpha_1(\iin{n+1})\\
  &\ge |\sin(\theta_n)|\la\alpha_1(\iin{n})
 \end{align*}
 and
 \begin{equation*}
  \|A_{\iin{n+1}}^T\vartheta_2(\iin{n})\|
  =\|A_{\ii(n+1)}^T A_{\iin{n}}^T\vartheta_2(\iin{n})\| \leq \alpha_1(\iii(n+1)) 
\|A_{\iin{n}}^T\vartheta_2(\iin{n})\| \leq \ua\alpha_2(\iin{n}).
 \end{equation*}
 To prove the claim we show that $\sum_{n=1}^\infty |\theta_n|$ converges by applying the root 
test. 
Since we have $|\sin(\theta_n)| \leq (\ua/\la) \alpha_2(\iin{n})/\alpha_1(\iin{n})$, it suffices to 
show that
 \[
  \limsup_{n\to\infty}\sqrt[n]{\alpha_2(\iin{n})/\alpha_1(\iin{n})}<1.
 \]
 This is equivalent to $\liminf_{n\to\infty}-\frac{1}{n} \log 
\alpha_2(\iin{n})/\alpha_1(\iin{n})>0$, which is true by our assumption $\lambda_2>\lambda_1$.
\end{proof}

\begin{remark}
  It should be remarked that in the Oseledets' theorem, compared to the setting of the iterated 
function systems, the matrices are iterated in the reverse order. This is why Lemma \ref{Oselemma} 
gives the convergence of the images of the singular directions, $\vartheta_k(\iii|_n)$, while the 
Oseledets' theorem concerns the singular directions $\eta_k(\iii|_n)$.
  
  The following example shows that in our setting, we can not expect the convergence of the 
singular 
directions in a set of positive measure. Let $\{A_i x + b_i\}_{i\in I}$ be an affine iterated 
function system. Suppose that for some $\jj\in I^*$ there exists $0<\alpha<1$ such that 
$A_\jj=\alpha R$, where $R$ is a rotation of angle $\theta\neq \pi$. Let $\mu$ be an ergodic 
measure 
on $\IN$ so that $\lambda_2>\lambda_1$ and $\mu[\jj]>0$. For any $\ii\in I^*$ the direction 
$\eta_1(\ii)$ differs from $\eta_1(\ii\jj)$ exactly by angle $\theta$. From \cite[Lemma 
2.3]{KaenmakiVilppolainen2008} we see that the set where $\jj$ occurs infinitely often is of full 
measure, so the set where $\eta_1(\iin{n})$ converges is of measure zero.
\end{remark}

\subsection{Zooming and patterns} \label{sec:zooming}
For $\ii\in \IN$ and $0<t<1$, let $n=n(\pi(\ii),t)$ be the largest integer for which the closed 
ball 
$B(\pi(\ii),t)$ only intersects one level $n$ construction cylinder $E_{\iin{n}}$. The existence of 
such $n$ is guaranteed by the strong separation condition. For reasons that will soon become 
apparent, this is called the {\bf construction level of the zoom}. It is easy to see that 
$n(\pi(\ii),t)$ increases as $t$ decreases to zero. We also abbreviate $n(\pi(\ii), t)$ by 
$n(\ii,t)$ and $\iin{n(\ii,t)}$ by $\iin{t}$.

For any vector $w\in \mathbb R^2 \setminus \{ 0 \}$, let $\proj_w$ be the orthogonal projection 
onto 
the line $\{tw : t\in \mathbb R\}$. Let $\jj\in I^*$. We use $h_{\iin{t}}(\jj)$ to denote the 
diameter of the projection $\proj_{\vartheta_1(\iin{t})} (E_{\iin{t}\jj})$. Similarly, 
$v_{\iin{t}}(\jj)$ is used to denote the diameter of $\proj_{\vartheta_2(\iin{t})}(E_{\iin{t} 
\jj})$. In the course of the proofs the construction cylinders will typically be written in the 
{\bf 
singular basis} $(\vartheta_1(\iin{t}),\vartheta_2(\iin{t}))$, turning the directions 
$\vartheta_1(\iin{t})$ and $\vartheta_2(\iin{t})$ horizontal and vertical, respectively, which is 
what the notations $h$ and $v$ stand for. 

\begin{figure}[t]
\includegraphics{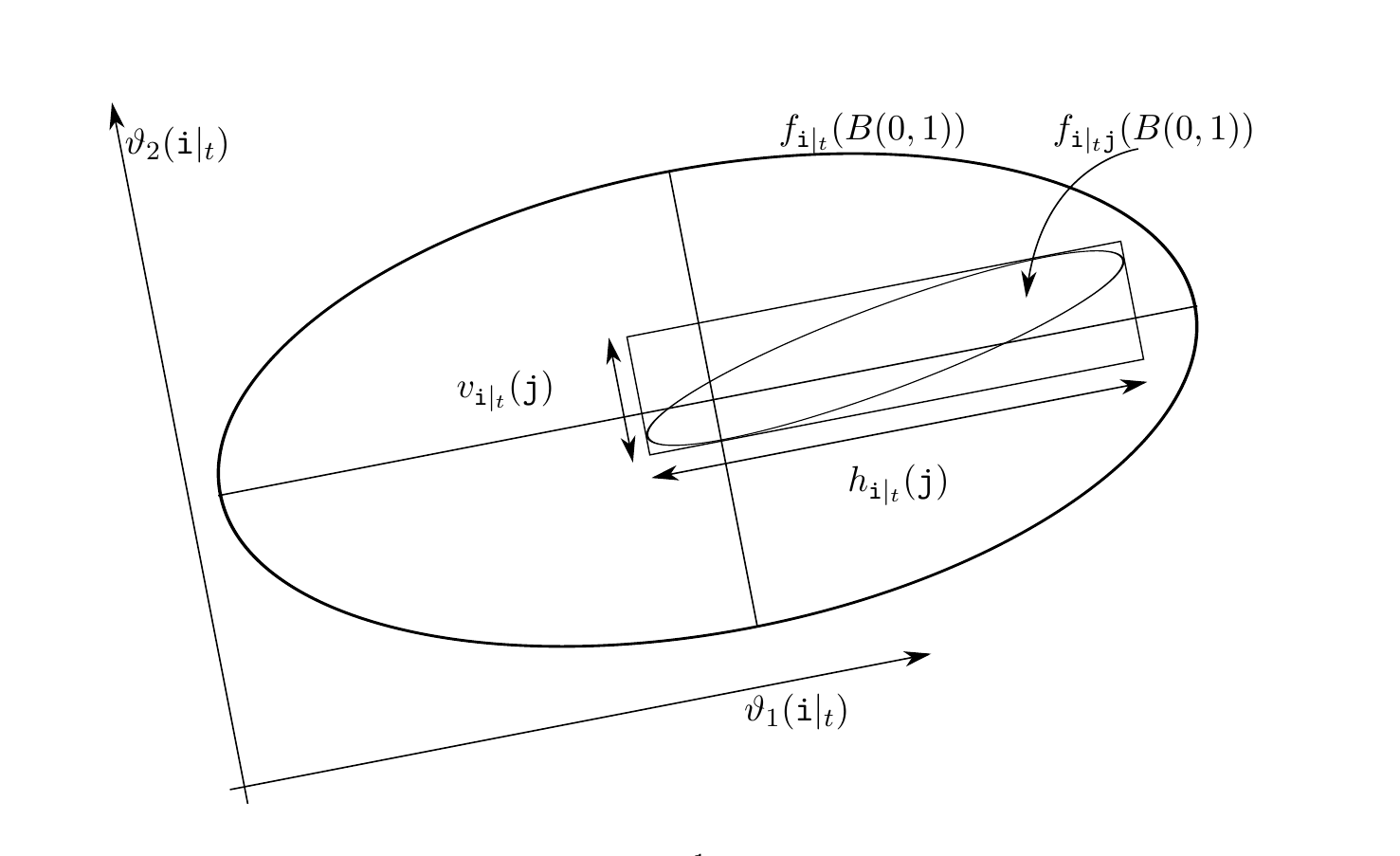}
\caption{The construction rectangle $R_{\iii|_t}(\jjj)$ is depicted in the picture. For 
illustrative 
purposes, the construction cylinders have been replaced by ellipses.}
\label{ellipsi}
\end{figure}

Notice that the construction cylinder $E_{\iin{t} \jj}$ is included in a closed rectangle with 
sides 
parallel to $\vartheta_1(\iin{t})$ and $\vartheta_2(\iin{t})$ and of side-lengths 
$h_{\iin{t}}(\jj)$ 
and $v_{\iin{t}}(\jj)$; see Figure \ref{ellipsi}. This {\bf construction rectangle} is denoted by 
$R_{\iin{t}}(\jj)$. For the empty word $\varnothing$ we use notation $R_{\iin{ t}}$. The next lemma 
highlights the relationship between the side-lengths of $R_{\iin{t}}(\jj)$ and the singular values 
of $A_{\iin{t}}$.

\begin{lemma} \label{heigthandlength}
If the self-affine set $E$ is not contained in a line, then there is a constant $L\in \mathbb N$ 
such that
\begin{equation*}
 L^{-1} \alpha_1(\iin{t})\la^{|\jj|} \leq h_{\iin{t}}(\jj)
 \qquad\text{and}\qquad
 v_{\iin{t}}(\jj) \leq \alpha_2(\iin{t})\ua^{|\jj|} L
\end{equation*}
for all $\ii\in\IN$, $\jj\in I^*$, and $t>0$. 
\end{lemma}
\begin{proof}
Since $E$ is compact and it is not contained in any line, there is a number $L\in\N$ and balls 
$B(x,L)$ and $B(y,L^{-1})$ so that $E$ is included in $B(x,L)$ and $B(y,L^{-1})$ is included in the 
convex hull of $E$. Therefore we have
$
 h_{\iin{t}}(\jj) \geq L^{-1}\alpha_1(\iin{t}\jj) \geq L^{-1}\alpha_1(\iin{t})\la^{|\jj|}
$
and
$
 v_{\iin{t}}(\jj) \leq L\alpha_2(\iin{t}\jj) \leq L\alpha_2(\iin{t})\ua^{|\jj|}
$
as claimed.
\end{proof}

Fix $x\in E$ and $t>0$. We define {\bf $t$-screen} at $x$ to be the closed ball $B(x,t)$ and the 
{\bf zoom function} $\ZZ_{x,t} \colon B(x,t)\to B(0,1)$ by setting
\[
  \ZZ_{x,t}(y)=\frac{y-x}{t}. 
\]
Then we define the {\bf scenery} $N_{x,t}$ around $x$ at scale $t$ by setting
$N_{x,t}=\ZZ_{x,t}(E\cap B(x,t))$. We consider distances of compact subsets of $\R^2$ in terms of 
the Hausdorff metric $d_H$. The notation $\dist$ is used for distances between points, or points 
and 
compact sets, in the Euclidean metric. Any limit of the sequence $(N_{x,t_n})$, where $t_n 
\downarrow 0$, is called a {\bf tangent set} of $E$ at a point $x$. We call a subset $S\subset 
B(0,1)$ an {\bf $\epsilon$-pattern}, if
\[
 S=\biggl( \R\times \bigcup_{i=1}^l I_i \biggr) \cap B(0,1),
\]
where $I_i$ are intervals of length less than
$\epsilon$ for all $i \in \{ 1,\dots, l \}$. We emphasize that these 
intervals are not assumed to be disjoint. Our goal is to study $\eps$-patterns coming from the 
construction rectangles -- even though the cylinders $E_\ii$ are disjoint, the construction 
rectangles might overlap.

Fix $\ii\in\IN$ and let $B(\ii,t)=B(\pi(\ii),t)$. Since the directions 
$\overline{\vartheta}_k(\ii)$ 
can differ for different $\ii$, we zoom into the set by applying, at each step, an appropriate 
rotation. To make this precise, consider the set in the singular basis, and define the {\bf rotated 
screen} $M_{\ii,t}$ by setting
\[
 M_{\ii,t}=\OO_{\ii,t}\left(\ZZ_{\pi(\ii), t}\left( B(\ii,t)\cap E \right)\right)= \OO_{\ii, 
t}(N_{\pi(\ii), t}).
\]
Here $\OO_{\ii,t}$ is the rotation that takes the singular basis $(\vartheta_1(\iin{t}), 
\vartheta_2(\iin{t}))$ to the standard basis of $\R^2$. Thus $M_{\ii, t}$ is the intersection of 
the 
set $E$ with a small ball around the point $\pi(\ii)$, scaled up to fill the whole unit ball, and 
turned around until $\vartheta_1(\iin{t})$ is horizontal. We also define the {\bf approximative 
sceneries} $P_{\ii,t}^K$ for all $K\in\N$ as 
\[
 P_{\ii,t}^K=\OO_{\ii,t}\biggl( \ZZ_{\pi(\ii), t}\biggl( B(\ii,t)\cap \biggl(\bigcup_{\jj\in 
R(\ii,t,K)} R_{\iin{t}}(\jj) \biggr) \biggr)\biggr) ,
\]
where $R(\ii,t,K)=\{\jj\in I^K: B(\ii,t)\cap E_{\iin{t}\jj}\neq\emptyset\}$. Thus $P_{\ii, t}^K$ 
tells us the level $K$ approximation of the set around $\pi(\ii)$ in the singular basis, after we 
have first zoomed into the scale $t$. Note that both $M_{\ii,t}$ and $P_{\ii,t}^K$ are subsets of 
$B(0,1)$. We say that $T\subset B(0,1)$ is a {\bf modified tangent} of $E$ at $\pi(\ii)$ if there 
exists a sequence $t_n\downarrow 0$ so that $M_{\ii,t_n}\to T$.

A set $A \subset \R^2$ is {\bf porous} if there exists $0<\alpha\le 1$ such that for every $x \in 
A$ 
and $0<r<\diam(A)$ there is a point $y \in A$ for which $B(y,\alpha r) \subset B(x,r) \setminus A$. 
We remark that a more precise name for this porosity condition is uniform lower porosity. It is 
well 
known that porous sets have zero Lebesgue measure; for example, see \cite[Proposition 
3.4]{JarvenpaaJarvenpaaKaenmakiRajalaRogovinSuomala2007}. Therefore a closed porous set is nowhere 
dense. Any ``fat'' Cantor set serves as an example of a closed nowhere dense set which is not 
porous.

\begin{lemma}\label{lem:convergence}
  If $T$ is a modified tangent of the self-affine set $E$, then $T$ is closed and porous.
\end{lemma}

\begin{proof}
  By \cite{Xi2008}, the set $E$ is porous and thus, by \cite[Proposition 5.6]{ChenRossi2015}, any 
tangent set is also porous. Since this remains true also for modified tangent sets we have finished 
the proof.
\end{proof}

Lemma \ref{Oselemma} says that the rotations $\OO_{\iii,t}$ converge for almost every $\ii\in\IN$. 
Therefore there is a correspondence between tangents and modified tangents in our setting. Thus all 
results obtained for modified tangent sets are valid for tangent sets and vice versa. In what 
follows, we mostly consider modified tangents because it is convenient to have a fixed orientation 
of the construction rectangles. 

\subsection{Idea of the proof} \label{sec:strategy}
The idea of the proof is to show that at almost every point of $E$, at all small scales, the 
approximative sceneries are $\epsilon$-patterns; see Lemma \ref{lem:bad set}. In this we are 
following Bandt and K\"aenm\"aki \cite{BandtKaenmaki2013}. In their situation all of the 
construction rectangles are uniformly flat in the vertical direction, but in our case it is not 
immediately clear what ``vertical'' even means, and flatness of the construction rectangles is not 
guaranteed in any direction. To deal with this difficulty we, first of all, let the screen rotate 
according to the singular basis with the construction level of the zoom, turning the construction 
rectangles so that they are flat in a controllable way in the rotated basis. 

On the other hand, to make sure that the construction rectangles are flat enough in the vertical 
direction of the singular basis, we prove in Lemmas \ref{good part} and \ref{height2} that there is 
a set of large measure so that the construction rectangles for points in this set are long in the 
horizontal direction and narrow in the vertical direction. Lemma \ref{good part} is proved below 
but 
Lemma \ref{height2} is postponed to the next section, as it requires some more definitions. 

To make use of Lemma \ref{lem:bad set}, we show in Lemma \ref{projectionlemma} that the 
approximative sceneries get close to the rotated screens and thus can be used to approximate the 
modified tangent sets. Here we do not know whether the rectangles in the approximative sceneries 
overlap or not, but this in not a problem, since by recalling Lemma \ref{lem:convergence}, we can 
deduce that the horizontal projection $C$ of the modified tangent set is porous. Finally, we use 
Lemma \ref{Oselemma} to transfer the result to the original tangent sets of $E$.

\begin{lemma}\label{good part}
If $\lambda_1<\theta<\gamma<\lambda_2$, then for all $\roo>0$ there is a set $E_\roo$ with 
$\nu_p(E_\roo)\ge 1-\roo$ such that the following two conditions are satisfied.
\begin{enumerate}
 \item\label{yksi}
 There are numbers $a=a(\gamma,\theta)>1$ and $N(\roo)\in \N$ such that for all $\ii\in E_\roo$ and 
$n\geq N(\roo)$ we have $\alpha_1(\iin{n})>a^n \alpha_2(\iin{n})$.
 
 \item\label{kaksi}
For all $D\in\N$ there is $N(D)\in \N$ such that for all $\ii\in E_\roo$, and for all $t$ with 
$n(\ii,t)=n,$ and for the $k\in \N$ satisfying $N(D)k \leq n < N(D)(k+1)$, we have 
$h_{\iin{t}}(\jj) 
> 2D \alpha_2(\iin{n})$ for all $\jj\in I^k$.
\end{enumerate}
\end{lemma}

\begin{proof}
(1) Fix $\roo>0$. By Egorov's theorem, we find a set $E_\roo\subset \IN$ with $\nu_p(\IN\setminus 
E_\roo)\leq\roo$ where $\lambda_{1}(\iin{n})$ and $\lambda_{2}(\iin{n})$ from \eqref{eq:sequences} 
converge uniformly. Thus we find $N=N(\roo)\in\N$ so that
\[
\alpha_1(\iin{n})\geq e^{-\theta n}
\qquad\text{and}\qquad
\alpha_2(\iin{n})\leq e^{-\gamma n}
\]
for all $n\geq N$. Letting $a=e^{\gamma - \theta}$ we have $\alpha_1(\iin{n})/\alpha_2(\iin{n}) \ge
e^{(\gamma-\theta)n}=a^n$, and the first claim is proved. 

(2) Let $L \in \N$ be as in Lemma \ref{heigthandlength}. Fix $D \in \N$ and choose $N(D) \geq 
N(\roo)$ so that
\begin{equation} \label{Lestimate}
  a^{N(D)} \la > 2DL.
\end{equation}
Let $\ii\in E_{\roo}$ and $t$ be so small that $ n :=n(\ii, t)\geq N(D)$. Fix $k\in\N$ so that 
$N(D)k \leq n < N(D)(k+1) $, and let $\jj\in I^k$. Now, by Lemma \ref{heigthandlength} and 
\eqref{yksi}, we have
\begin{align*}
h_{\iin{t}}(\jj)
&\geq \alpha_1(\iin{t})\la^{k} L^{-1}
 > a^n \alpha_2(\iin{t})\la^{k} L^{-1}\\
&\geq \alpha_2(\iin{t}) (a^{N(D)} \la)^k L^{-1}
>2^kD^k L^{k-1} \alpha_2(\iin{t})
 \geq 2D \alpha_2(\iin{t})
\end{align*}
as claimed.
\end{proof}

\section{Main result}\label{sec:main}

Let us begin by formulating the main theorem of the article.

\begin{theorem} \label{main}
  Suppose that $\{ f_1,\ldots,f_m \}$ is an affine iterated function system and $E$ the associated 
self-affine set. If
  \begin{enumerate}
  \item\label{as:separation} there exists $\delta>0$ is such that $\min\{ \dist(f_i(E),f_j(E)) : i 
\ne j \} > \delta$; that is, $E$ satisfies the strong separation condition,
  \item\label{as:projection} for $\nu_p$-almost all $\ii\in\IN$ there is $n_0 \in \N$ such that for 
all $n \ge n_0$ and all $\jj\in I^*$, the projection $\proj_{\overline\vartheta(\ii)} 
(E_{\iin{n}\jj})$ is a line segment,
  \item\label{as:lyapunov} the probability vector $p=(p_1,\dots,p_m)$ is such that the Lyapunov 
exponents satisfy $\lambda_1<\lambda_2$,
  \end{enumerate}
  then for $\nu_p$-almost all $x\in E$ the tangent sets at $x$ are either of the form 
$\OO((\R\times 
C) \cap B(0,1))$, where $C$ is a closed porous set, or of the form $\OO((\ell\times \{0\}) \cap 
B(0,1))$, where $\ell$ is an interval containing at least one of the intervals $[-1,0]$ and 
$[0,1]$. 
Here $\OO$ is the rotation that takes $(1,0)$ to $\overline{\vartheta}_1(\pi^{-1}(x))$.
\end{theorem}

\begin{remark} \label{re:ontheassumptions2}
  To verify the assumption \eqref{as:lyapunov} in Theorem \ref{main}, it suffices to check that the 
iterated function system is pinching and twisting since then the assumption \eqref{as:lyapunov} 
follows immediately from \cite[Theorem 1.2]{Viana2014}; see also \cite{Furstenberg1963}. An affine iterated function system is {\bf 
pinching} if for any constant $C>1$ there is a finite word $\ii$ so that 
$\alpha_1(\ii)>C\alpha_2(\ii)$. It is {\bf twisting} if for any finite set of vectors 
$\{v,w_1,\ldots,w_n\}\subset\R^2$, there exists a finite word $\jj$ so that $A_\jj v$ is not 
parallel to $w_j$ for any $j\in\{1,\ldots,n\}$. It is worthwhile to remark that
  in particular \cite[Theorem 1.2]{Viana2014} applies to any Bernoulli measure obtained from a positive 
probability vector.
  
  In the carpet case, where the linear part of $f_i$ is the diagonal matrix $\diag(h_i,v_i)$, the 
assumption \eqref{as:lyapunov} is equivalent to
  \begin{equation} \label{eq:3.1}
    -\sum_{i=1}^m \nu_p(f_i([0,1]^2)) \log h_i \ne -\sum_{i=1}^m \nu_p(f_i([0,1]^2)) \log v_i.
  \end{equation}
  Indeed, by the ergodic theorem, the left-hand side of \eqref{eq:3.1} equals to $\lim_{n \to 
\infty} \tfrac{1}{n} \log h_{\iii|_n}$ and the right-hand side equals to $\lim_{n \to \infty} 
\tfrac{1}{n} \log v_{\iii|_n}$ for $\nu_p$-almost all $\iii \in I^\N$. Fix $\iii$ so that these 
limits and $\lambda_1(\iii)$, $\lambda_2(\iii)$, and $\overline{\vartheta}_1(\iii)$ exist. By Lemma 
\ref{Oselemma}, $\vartheta_k(\iii|_n)$ converges to $\overline{\vartheta}_k(\iii)$ for both $k \in 
\{ 1,2 \}$. Therefore, for some $k \in \{ 1,2 \}$, we have $\alpha_k(\iii|_n) = h_{\iii|_n}$ and 
$\alpha_{3-k}(\iii|_n) = v_{\iii|_n}$ for all large enough $n$. This clearly implies that 
$\lambda_k 
= -\sum_{i=1}^m \nu_p(f_i([0,1]^2)) \log h_i$ and $\lambda_{3-k} = -\sum_{i=1}^m 
\nu_p(f_i([0,1]^2)) 
\log v_i$.

From the point of view of our theorem, it would be interesting to know if for any self-affine set there exists a positive probability vector $p$ that gives rise to distinct Lyapunov exponents. Definitely pinching is a necessary condition for 
this: having distinct Lyapunov exponents imply that $\alpha_1(\iin{n})$ gets exponentially larger than $\alpha_2(\iin{n})$ at $\nu_p$ almost everywhere. Since it is easy to define Bernoulli measures having distinct Lyapunov exponents on self-affine carpets we see that twisting is not a necessary condition. We also remark that there exist affine iterated function systems where the mappings are not similitudes but the Lyapunov exponents coincide for all ergodic measures. For example, choose mappings that have the same linear part which is a composition of a diagonal contraction having distinct eigenvalues and a rotation of 90 degrees. Since the second level compositions of the mappings are similitudes it is impossible to have $\lambda_1<\lambda_2$ for any ergodic measure.
\end{remark}

\begin{remark} \label{re:ontheassumptions}
  The assumption \eqref{as:projection} in Theorem \ref{main} is referred to as the {\bf projection 
condition}. It is satisfied for example if the projection of the the set $\bigcup_{i=1}^m f_{i}(X)$ 
in any direction is a line segment, where $X$ is the convex hull of $E$. To see this, fix a line 
$\ell$ that intersect $\bigcup_{i=1}^m f_i(X)$. Let $j_1$ be such that $\ell$ intersects 
$f_{j_1}(X)$. The crucial observation now is that the line $f_{j_1}^{-1}(\ell)$ intersects 
$\bigcup_{i=1}^m f_i(X)$. If this was not the case, then the line $f_{j_1}^{-1}(\ell)$ would divide 
the convex hull $X$ in two parts, both of which contain sets $f_i(X)$. This contradicts the 
assumption. Therefore, if $j_2$ is so that $f_{j_1}^{-1}(\ell)$ intersects $f_{j_2}(X)$, then we 
see 
that $\ell$ intersects $f_{j_1} \circ f_{j_2}(X)$. Continuing in this manner, we find $\jjj = 
(j_1,j_2,\ldots) \in I^\N$ such that $\pi(\jjj) \in \ell$ which is what we wanted to show.
  
  It is worth noticing that in the carpet case, if \eqref{eq:3.1} holds, then it suffices to 
consider only one projection: There are exactly two singular directions which are invariant under 
all the maps, so that in any case, in order for the assumption \eqref{as:projection} to hold, it 
suffices to consider at most two directions. We may assume that the right-hand side of 
\eqref{eq:3.1} is greater than the left-hand side. This guarantees that the horizontal direction is 
the direction $\overline\vartheta_1(\ii)$ for $\nu_p$-almost all $\ii\in \IN$. Indeed, notice that 
for almost all $\ii\in \IN$, eventually $\diam(\proj_x(R_{\iin{n}}))>\diam(\proj_y(R_{\iin{n}}))$, 
where $\proj_x$ and $\proj_y$ denote the projections onto the $x$-axis and $y$-axis, respectively. 
This means that for almost all $\ii\in \IN$, eventually $\diam(\proj_x(R_{\iin{n}}))=\alpha_1 
(\iin{n})$. Hence the vector $\vartheta _1(\iin{n})$ eventually becomes horizontal. Observe that it 
is essential that the projection condition is defined pointwise.
\end{remark}

\begin{remark}\label{rem:noline}
  The assumptions \eqref{as:separation} and \eqref{as:projection} of Theorem \ref{main} imply that 
$E$ is not contained in a line: Assume, to the contrary, that $E$ is contained in a line $\ell$. 
Then for any $n$ all the rectangles $R_{\iin{n}}$ intersect this line. Condition 
\eqref{as:separation} guarantees that $E$ itself cannot be a line segment. Thus $E_{\iin{n}}$ is 
not 
a line segment. Also, the direction $\overline{\vartheta}_2(\ii)$ is not the direction of the line 
$\ell$ for any $\ii$ in a set of full measure, since then the projection onto 
$\overline{\vartheta}_1(\ii)$ would be a single point and not a line segment. Thus, for almost all 
$\ii$, the angle between $\overline{\vartheta}_2(\ii)$ and $\ell$ is positive, implying that  
$\proj_{\overline{\vartheta}_1(\ii)} (E_{\iin{n}\jj})$ is not a line segment thus giving a 
contradiction with the assumption \eqref{as:projection}.

  Therefore, since $E$ is not contained in any line, there is a construction level $k_0$ so that 
when $|\ii|=k_0$, no line in any direction intersects all ellipses $f_\ii(B)$, where $B$ is a ball 
containing $E$. Without loss of generality we may assume that $k_0=1$, since otherwise we can 
consider the iterated function system $\{f_\ii:\ii\in I^{k_0}\}$. This is not a restriction, since 
tangent sets of $E$ only depend on the set $E$ itself and not on the iterated function system that 
generates $E$.
\end{remark}

We will now start preparations for the proof of the main result. Working in the setting of Lemma 
\ref{good part}, we introduce a sequence of lemmas gradually converging to the proof. Fix $\roo>0$ 
and let $E_\roo$ be as in Lemma \ref{good part}. Furthermore, let $K\in \mathbb N$ and denote 
\[
\epsilon(K)=  L \delta^{-1}\ua^K,
\] 
where $L$ is as in Lemma \ref{heigthandlength} and $\delta>0$ as in the assumption 
\eqref{as:separation} of Theorem \ref{main}.

\begin{lemma} \label{height2}
  Under the assumptions of Theorem \ref{main}, for every $K\in\N$ there exists $t_0>0$ such that 
$v_{\iin{t}}(\jj)\leq t \epsilon(K)$ for all $t<t_0$, $\ii\in E_\roo$, and $\jj\in I^K$.
\end{lemma}

\begin{proof}
The $t$-screen $B(\pi(\ii),t)$ intersects at least two level $n(\ii,t)+1$ construction cylinders, 
and by assumption \eqref{as:separation} of Theorem \ref{main}, we have the estimate $t\geq 
\alpha_2(\iin{t}) \delta \geq \la^{n(\ii,t)} \delta$, where $\delta$ is as in that assumption. This 
shows that we have a uniform lower bound for $n(\ii,t)$ that increases as $t$ decreases. Take 
$t_0>0$ so small that for all $t<t_0$, it is the case that $n(\ii, t)>K$ for all $\ii\in \IN$. By 
Lemma \ref{heigthandlength}, we then have
 \[
  v_{\iin{t}}(\jj)\leq \alpha_2(\iin{t}) \ua^K L \leq t\delta^{-1}\ua^K L= t \epsilon(K)
 \]
 as claimed.
\end{proof}

As described in \S \ref{sec:strategy}, we want to investigate the size of the set of points for 
which the approximative scenery is not an $\epsilon(K)$-pattern. For technical reasons we consider 
the following sets. For every $D>0$ and $K, n\in \mathbb N$ we define 
\begin{align*}
  A_{n,K}(D) = \{\ii\in E_\roo : \;&\text{there is }t\text{ such that }n(\ii, t)=n, \\ &D 
\alpha_2(\iin{n}) >t, \text{ and  }P_{\ii, t}^K \text{ is not an }\epsilon(K)\text{-pattern}\},
\end{align*}
and let $B_K(D)= \limsup_{n \to\infty} A_{n,K}(D)$.

In the following two lemmas, the reader should bear in mind that while the strong separation 
condition guarantees the construction cylinders to be disjoint, the corresponding construction 
rectangles may overlap. This is not a problem since, for example, the proof of the following lemma 
concerns only the number of certain vertical edges. Recall also that the choice of $n(\iii,t)$ 
guarantees that the screen $B(\pi(\iii),t)$ contains points only from one level $n$ construction 
cylinder.

\begin{lemma}\label{lem:bad set}
  Under the assumptions of Theorem \ref{main}, we have $\nu_p(B_K(D))=0$ for all $D>0$ and $K\in 
\mathbb N$.
\end{lemma}

\begin{proof}
  Fix $D$ and $K$, and let $N(D)$ be as in Lemma \ref{good part}. Our plan is to prove that 
  \begin{equation*}
  \sum_{n=1}^\infty \nu_p(A_{n,K}(D))<\infty
  \end{equation*}
  since then the claim follows from the Borel-Cantelli lemma. To that end, we estimate the measure 
of the sets $A_{n,K}(D)$. 

  \begin{figure}[t]
  \includegraphics{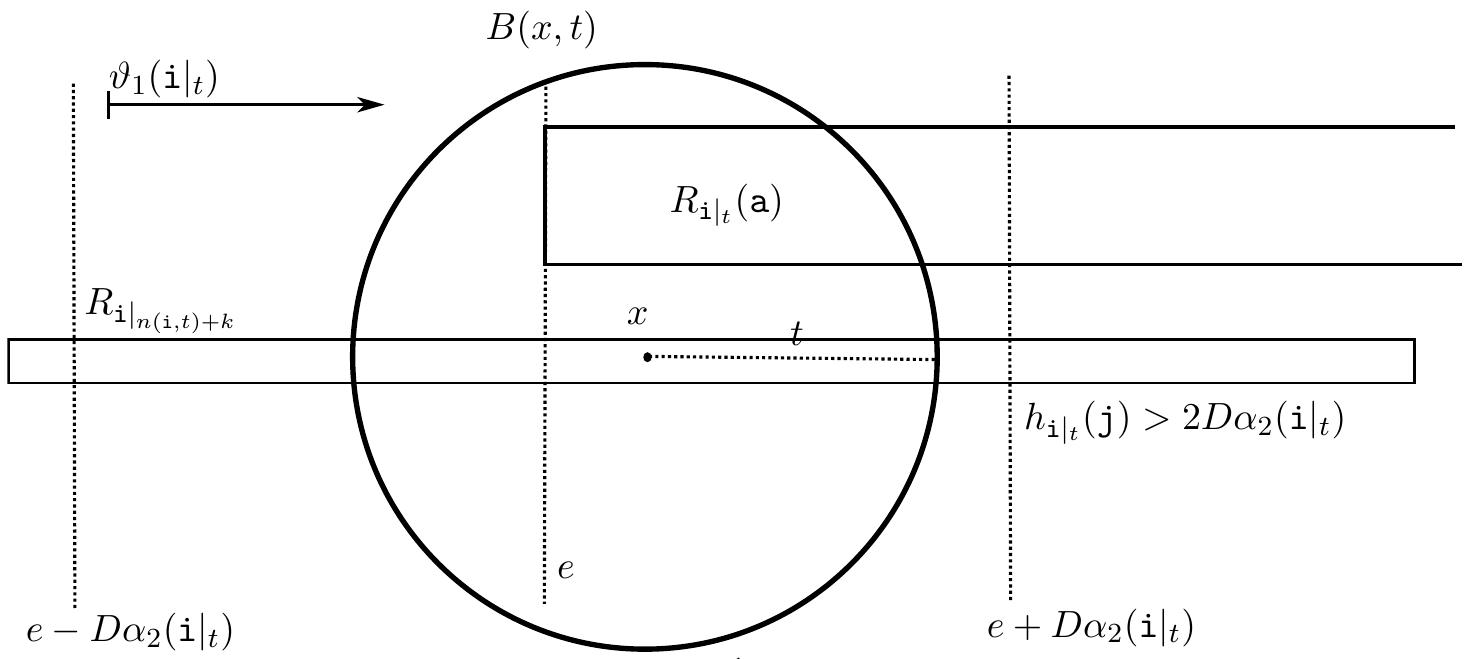}
  \caption{The point $x$ belongs to a forbidden rectangle because within distance $t$ away in the 
horizontal direction an endpoint $a$ of level $n+K$ construction rectangle $R_{\ii|_t}(\aa)$ 
appears.}
  \label{genforb}
  \end{figure}

  We will cover the set $A_{n,K}(D)$ by construction cylinders corresponding to forbidden words 
which will be defined shortly. Throughout, we are considering the situation in the singular basis 
and thus, we shall refer to the directions $\vartheta_1(\iin{n})$ and $\vartheta_2(\iin{n})$ as 
horizontal and vertical, respectively. This should not be a cause of confusion, as the basis in use 
is clear from the context. The {\bf forbidden words} are defined in the following way: For any 
$k\in\N$ and $|\ii|=n$ with $N(D) k\leq n< N(D) (k+1)$ the word $\ii\jj$ is $n$-forbidden, if 
$|\jj|=k$ and $R_{\ii}(\jj)$ intersects any of the vertical line segments $\{ c \} \times \R$ where 
  \begin{align*}
    c\in\{e \pm D\alpha_2(\ii) : \;&e\text{ is an $x$ coordinate of any of the $2m^K$} \\
    &\text{vertical edges of the rectangles } R_{\ii}(\aa) \text{ with } \aa\in I^K\}.
  \end{align*}
  Now fix a point $x=\pi(\uu)\in\pi (A_{n,K}(D))$. Then there is $t$ such that $n(x,t)=n$ and 
$P^K_{\uu,t}$ is not an $\epsilon(K)$-pattern. By Lemma \ref{height2}, for each $|\aa|=K$, the 
rectangles $R_{\uu|_n}(\aa)$ have height at most $\epsilon(K)t$. Thus the only way the 
approximative 
scenery $P^K_{\uu,t}$ around $x$ is not an $\epsilon(K)$-pattern is if some endpoint of a rectangle 
is in the $t$-screen $B(x,t)$. More precisely, this means that within distance $t$ from $x$ in the 
horizontal direction, there is an endpoint of a rectangle from level $n+K$. By Lemma \ref{good 
part}\eqref{kaksi} and the definition of $A_{n,K}(D)$,
  \[
    h_{\uu|_{n}}(\uu(n+1),\ldots,\uu(n+k)) > 2D \alpha_2(\uu|_n) > D \alpha_2(\uu|_n) >t,
  \]
  so that in this case the rectangle $R_{\uu|_n}(\uu(n+1),\ldots,\uu(n+k))$ necessarily intersects 
one of the forbidden line segments and hence $\uu|_{n + k}$ is an $n$-forbidden word. See Figure 
\ref{genforb}. By this argument, we see that $\pi(A_{n,K}(D))$ is covered by $n$-forbidden 
rectangles, as claimed. 

  Now, let $\ii\in I^n$. By Remark \ref{rem:noline}, no line in any direction in the rectangle 
$R_{\ii}$ intersects all the sub-rectangles $R_{\ii}(j)$ of level $n+1$. Therefore, the relative 
total mass of the sub-rectangles $R_{\ii}(j)$ that the line intersects in the vertical direction is 
at most $(1-\underline p)$. By the self-affinity and properties of the Bernoulli measures, we see 
that the relative total mass of the level $n+2$ sub-rectangles, which the vertical line intersects 
is at most $(1-\underline p)^2$. Iterating this, and noticing that there are $4m^K$ forbidden line 
segments, we get that
  \begin{align}
  \label{forbiddenmeasure}
  \nu_p(A_{n,K}(D))
  &\leq \sum_{\ii\jj\textrm{ is $n$-forbidden}}\nu_p(E_{\ii\jj})
    \leq \sum_{\ii\in I^n}\nu_p(E_\ii) (1-\underline p)^{k} 4m^K
    = (1-\underline p)^{k} 4m^K
  \end{align}
  for $N(D)k\leq n <N(D)(k+1)$. Thus
  \[
  \sum_{n=1}^{\infty} \nu_p(A_{n,K}(D)) \leq \sum_{n=1}^{N(D)-1} \nu_p(A_{n,K}(D))
  +N(D)\frac{(1-\underline p) 4m^K}{1-(1-\underline p)}<\infty
  \]
  finishing the proof.
\end{proof}

\begin{lemma} \label{projectionlemma}
  Under the assumptions of Theorem \ref{main}, for every $K \in \N$ and for almost every $\iii \in 
E_\roo$ there exists $t_0>0$ such that $d_H(P^K_{\ii,t},M_{\ii,t})<  5 \sqrt{\eps(K)}$ for all 
$0<t<t_0$.
\end{lemma}

\begin{proof}
  Let $\ii\in\IN$ be generic in the sense that $\overline \vartheta_1(\ii)$ from Lemma 
\ref{Oselemma} exists. Now $t_0>0$ is determined from $n_0$ of the assumption \eqref{as:projection} 
in Theorem \ref{main} and from Lemmas \ref{Oselemma} and \ref{height2} so that for all $0<t<t_0$ 
and 
$\jjj \in I^*$ the projection $\proj_{\overline{\vartheta}(\iii)}(E_{\iii|_t \jjj})$ is a line 
segment, $\vartheta_1(\iii|_t)$ is close to $\overline{\vartheta}_1(\iii)$, and the height of 
$R_\iii(\jjj)$ is at most $\eps(K)$ for all $\jjj \in I^K$. Notice that $M_{\ii, t}\subset P^K_{\ii, 
t}$. Fix a point $x\in P^K_{\ii, t}$ and let $\jjj \in I^K$ be such that $x \in \OO_{\ii, 
t}\ZZ_{\pi(\ii), t} R_{\iin{ t}}(\jj)$. We now wish to prove that within $5 \sqrt{\eps(K)}$ away 
from $x$ there is a point $y\in M_{\ii, t}$. The study divides into three separate cases; see Figure 
\ref{Lemma37}.

\begin{figure}[t]
\includegraphics{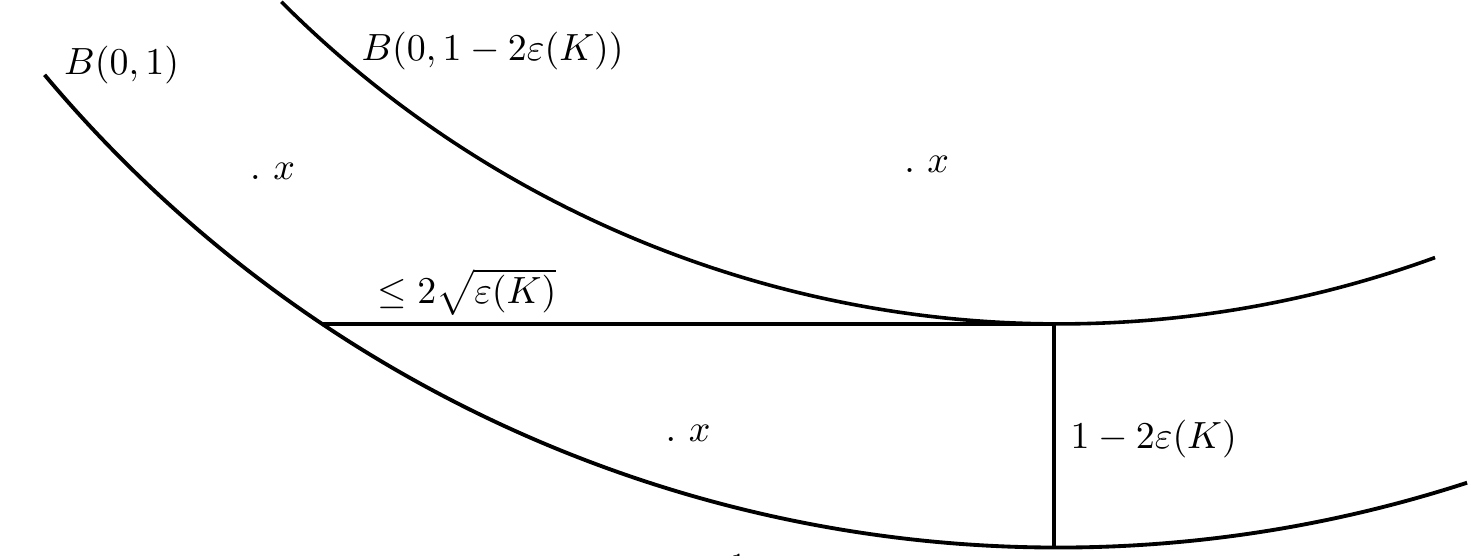}
\caption{The picture indicates the possible positions of $x$ in the proof of Lemma 
\ref{projectionlemma}.}
\label{Lemma37}
\end{figure}

  If $x$ is far from the boundary of the screen, we can find the point $y$ on the line $\ell$ 
crossing through $\ZZ_{\pi(\ii),t}^{-1}(x)$ in direction $\overline\vartheta_2(\ii)$. Assume that 
$x\in B(0,1-2\eps(K))$, say, and denote by $\theta$ the angle between $\overline \vartheta _1(\ii)$ 
and $\vartheta_2(\iin{t})$. By Lemma \ref{Oselemma}, we may assume that $\sin(\theta)\ge 
\frac{1}{2}$. By assumption \eqref{as:projection} of Theorem \ref{main}, there exists $z\in 
\ell\cap 
E_{\iin{t}\jj}$, as the projection $\proj_{\overline \vartheta_1(\ii)}(E_{\iin{t}\jj})$ is 
connected. By Lemma \ref{height2}, we get that $v_{\iin{t}}(\jj)\leq t\eps(K)$. Thus 
$\dist(\ZZ_{\pi(\ii),t}^{-1}(x),z)\leq 2t\eps(K)$, and so $z\in B(\pi(\ii),t)$ by the triangle 
inequality. When zooming out, we get $\dist(x,y)\leq 2\eps(K)$ for the point 
$y=\ZZ_{\pi(\ii),t}(z)\in M_{\ii, t}$.

If $x$ is very close to the boundary of the screen, it is possible that the above reasoning gives 
a point $y$ which is not inside the screen. These points need to be dealt differently. Consider 
first the case $\OO_{\ii, t}\ZZ_{\pi(\ii), t} R_{\iin{ t}}(\jj) \cap B(0,1-2\eps(K))=\emptyset$. 
Let 
$y \in M_{\ii, t} \cap \OO_{\ii, t}\ZZ_{\pi(\ii), t} R_{\iin{ t}}(\jj)$ be arbitrary -- such a 
point 
exists by the definition of $P^K_{\ii, t}$. Through estimating the length of line segments 
contained 
in the annulus $B(0, 1)\setminus B(0, 1-2\eps(K))$, we can bound the horizontal distance between 
$x$ 
and $y$ by $4 \sqrt{\eps(K)}$. Thus $\dist(x,y)\leq 4 \sqrt{\eps(K)}+\eps(K)$.
  
To finish the considerations, let $x\not\in B(0,1-2\eps(K))$ but assume that $\OO_{\ii, 
t}\ZZ_{\pi(\ii), t} R_{\iin{ t}}(\jj)\cap B(0,1-2\eps(K))\neq\emptyset$. Then there is a point 
$x'\in B(0,1-2\eps(K))$ so that, as for $x$ and $y$ above, $\dist (x, x')\le 2\sqrt{\eps(K)}$. For 
this $x'$ we find $y\in M_{\ii, t}$ as in the first case, with $\dist(x',y)\leq 2 
\eps(K)$. Thus $\dist (x, y)\le 2\sqrt{\eps(K)} + 2\eps(K)$. Since $\eps(K)\leq\sqrt{\eps(K)}$ we 
have now finished the proof.
\end{proof}

We now combine the above lemmas to prove Lemma \ref{lem:modifiedtangents}. After that we are ready 
to prove the main theorem. 

\begin{lemma}\label{lem:modifiedtangents}
  Under the assumptions of Theorem \ref{main}, for $\nu_p$-almost all $\iii \in E_\roo$, the 
modified tangent sets at $\pi(\iii)$ are either of the form $(\R\times C)\cap B(0,1)$ where $C$ is 
a 
closed porous set, or of the form $(\ell\times \{0\})\cap B(0,1)$, where $\ell$ is an interval 
containing at least one of the intervals $[-1,0]$ and $[0,1]$.
\end{lemma}

\begin{proof}
  Recall that by Lemma \ref{lem:bad set}, the set $B(M)= \bigcup_{K=0}^\infty B_K(M)$ has zero 
measure for all $M\in \N$. Fix a point $\ii \in E_\roo\setminus \bigcup_{M=1}^\infty B(M)$, and a 
modified tangent set $T=\lim_{k\to\infty}M_{\ii, t_k}$ at $x=\pi(\ii)$. We assume that $\ii$ is 
generic, in the sense that $\overline\vartheta_1(\ii)$ from Lemma \ref{Oselemma} exists.

  Now, there are two options. Either 
  \begin{equation} \label{tozero}
    \lim_{k\to \infty} \frac{ \alpha_{2}(\iin{ n(\ii,t_k) } )}{t_k}= 0,
  \end{equation}
  or there is $D>0$ and infinitely many $k$ such that 
  \begin{equation} \label{nottozero}
    \frac{ \alpha_{2}(\iin{ n(\ii,t_k) }) }{t_k}>\frac {1}{D}. 
  \end{equation}
  We will first consider the situation where \eqref{tozero} holds. For the time being, fix $t>0$ 
and 
the corresponding $n=n(\ii, t)$ so that $n-1 \ge n_0$, where $n_0$ is as in the assumption 
\eqref{as:projection} of Theorem \ref{main}, and assume that both endings of the construction 
rectangle $R_{\iin{n}}$ can be seen in the screen. It might be that one of the endings of 
$\proj_{\overline{\vartheta}_1(\ii)}(E_{\iin{n}})$ is at the end of the line segment 
$\proj_{\overline{\vartheta}_1(\ii)}(E_{\iin{n-1}})$, but not both of them. Thus, by the assumption 
\eqref{as:projection} of Theorem \ref{main}, for at least one of the endpoints of $R_{\iin{n}}$, a 
line $\ell$ in the direction $\overline{\vartheta}_2(\ii)$ through the endpoint necessarily also 
intersects another cylinder $E_{\iin{n - 1}j}$ for some $j \in \{ 1, \ldots, m \}$. This is the 
case, since otherwise there is a hole in the projection of $E_{\iin{n-1}}$ onto the line in 
direction $\overline{\vartheta}_1(\ii)$. 

  Let $d$ denote the distance from $x$ to that ending of the rectangle $R_{\iin{n}}$ and let $e$ 
denote the distance between $E_{\iin{n}}$ and $E_{\iin{n-1}j}$ along the line $\ell$. Notice that 
$e$ is bounded from above by $v_{\iin{n-1}}/\cos \theta_n$, where $\theta_n$ is the angle between 
$\vartheta_2(\iin{n-1})$ and $\overline \vartheta_2(\ii)$. By Lemma \ref{Oselemma}, taking $t$ 
small 
enough we may assume that $\cos\theta_n\geq 1/2$, and hence $e\le 2v_{\iin{n-1}}$. We want to show 
that $d/t \to 1$ as $t\to 0$. Assume this is not the case, that is, $d<ct$ for some $c<1$. Then, 
since the screen $B(x,t)$ intersects only the cylinder $E_{\iin{n}}$, using Lemma 
\ref{heigthandlength},
  \[
  t\le \dist(x, E_{\iin{n-1}j})\le d+v_{\iin{n}} + e\le ct + v_{\iin{n}}+ 2v_{\iin{n-1}} \le ct + 
L(1+2\ua)\alpha_2(\iin{n}).
  \]
  This gives $t\le C\alpha_2(\iin{n})$ for some absolute constant $C>0$. Thus, for all small $t$, 
  \[
  \frac{\alpha_2(\iin{n(\ii, t)})}{t}>0. 
  \]
  These observations mean that if \eqref{tozero} holds, then either both endings of the 
construction 
rectangle $R_{\iin{t_k}}$ can be seen in the screen for only finitely many $k$, or at least one of 
the endpoints of the rectangles reaches the boundary of the screen $B(x, t_k)$ in the limit. 

  In the latter case, using Lemma \ref{heigthandlength}, $P_{x,t}^0$ consists of a strip with 
height 
converging to zero. By the argument of Lemma \ref{projectionlemma} and compactness, the modified 
tangent set $T$ is a horizontal line segment, containing at least one of the line segments 
$[1,0]\times\{0\}$ and $[-1,0]\times\{0\}$. 

  Assume now that at most one of the endings is seen in the screen, apart from maybe finitely many 
$(t_k)$. Since there are only two endings, there is a sub-sequence of $(t_k)$ so that along that 
sub-sequence, either the left or the right endpoint is never in the screen, and we can use the 
argument from the previous paragraph to deduce the same claim. Notice that because the limit set 
$T$ 
exists, it is unique, and thus it suffices to prove the convergence along a sub-sequence. 

  Let us then assume that \eqref{nottozero} holds, that is, $D \alpha_2(\ii|_{t_k})>t_k$ along a 
sub-sequence which we keep denoting by $(t_k)$. Let $M$ be an integer with $M\ge D$ and notice that 
$\ii\notin B(M)$. Fix $K$, and notice that also $\ii\notin B_K(M)$, so that $\ii\in A_{n,K}(M)$ for 
at most finitely many $n\in \N$. 

  For all $k$ it is the case that $M\alpha_2(\ii|_{t_k})\ge D \alpha_2(\ii|_{t_k})>t_k$, so that if 
there are only finitely many $k$ such that $P^K_{\ii, t_k}$ would be an $\epsilon(K)$-pattern, then 
$x\in A_{n(x,t_k), K}(M)$ for infinitely many $k$. Since $t_k\to 0$ implies $n(x, t_k)\to \infty$, 
this gives infinitely many $A_{n, K}$'s, which cannot be the case. Hence for all $K$ there is a 
$t_K$ such that $P^K_{\ii, t_K}$ is an $\epsilon(K)$-pattern and that Lemma \ref{projectionlemma} 
holds. Since there is a sequence of $\eps$-patterns converging to the modified tangent set $T$, it 
must be of the form $(\R \times C) \cap B(0,1)$ for some set $C \subset [-1,1]$. Since, by Lemma 
\ref{lem:convergence}, $T$ is closed and porous in $B(0,1)$, the same must hold for $C$ in $[-1,1]$.
\end{proof}

We are now ready to prove Theorem \ref{main}. 

\begin{proof}[Proof of Theorem \ref{main}]
  So far, in the previous lemmas, we have verified the claims for almost all points in the set 
$E_\roo$ having measure at least $1-\roo$. Since the claim of Lemma \ref{lem:modifiedtangents} does 
not depend on $\roo$ it actually holds for almost all points in $E$: If there is an exceptional set 
of positive measure, then we can repeat the argument for some $\roo$ smaller than half of the size 
of the exceptional set to get a contradiction.

  Let us now assume that $T$ is a tangent set of $E$ at $x=\pi(\ii)$, along a sequence $(t_n)$. By 
the above discussion, we may assume that Lemma \ref{lem:modifiedtangents} holds at $x$. By 
compactness, we find a sub-sequence, also denoted by $(t_n)$, so that $M_{\ii,t_n}\to F$ along this 
sub-sequence. By Lemma \ref{lem:modifiedtangents}, we know that $F$ is either of the form 
$(\R\times 
C) \cap B(0,1)$, where $C$ is a closed porous set, or of the form $(\ell\times \{0\}) \cap B(0,1)$ 
where $\ell$ is an interval with $[-1,0]\subset \ell$ or $[0,1]\subset \ell$. Let $\OO$ be the 
rotation taking $\overline{\vartheta}_1(\ii)$ to $(1,0)$, where $\overline\vartheta_1(\ii)$ is from 
Lemma \ref{Oselemma}. 

  It suffices to show that $T=\OO^{-1}F$. Since $\vartheta_1(\iin{t_n}) \to 
\overline{\vartheta}_1(\ii)$, the rotations $\OO_{\ii,t_n}$ converge to $\OO$. Let $\epsilon>0$ and 
choose $n_0$ so that $d_H( \OO T,\OO_{\ii,t_n}T)\leq \epsilon$, $d_H(T , N_{\ii,t_n})\leq\epsilon$, 
and $d_H(M_{\ii,t_n}, F)\leq \epsilon$ for all $n\geq n_0$. By the triangle inequality, we then 
have 
  \begin{align*}
    d_H(\OO T,F)
    &\leq d_H( \OO T , \OO_{\ii,t_n}T)
    + d_H( \OO_{\ii,t_n}T , M_{\ii,t_n})
    +d_H(M_{\ii,t_n} , F)\leq 3\epsilon
  \end{align*}
  which completes the proof.
\end{proof}

\section{Discussion}\label{sec:discussion}

We formulated our main theorem by using assumptions as general as possible. In Remarks 
\ref{re:ontheassumptions2} and \ref{re:ontheassumptions}, we provided the reader with sufficient 
and 
checkable conditions. In this section, we continue this discussion by examining the effect of some 
of the assumptions and exhibiting examples. We also prove Corollary \ref{cor:carpet} stated in the 
introduction.

We say that a subset of the self-affine set $A\subset E$ satisfies the {\bf line condition}, if 
there exists $N\in \N$ such that for all $\pi(\ii) \in A$ and for all large $n$, any line in 
direction $\vartheta_2(\iin{n-N})$ that intersects $E_{\iin{n}}$ also intersects $E_{\iin{n-N}\jj}$ 
for some $\jj\in I^*$ that satisfies $\ii(n-N+1)\neq\jj(1)$.

\begin{proposition}\label{cor:nopoint}
  Under the assumptions of Theorem \ref{main}, if for all $\roo>0$ there is a subset $E_\roo 
\subset 
E$ with $\nu_p(E_\roo) \ge 1 - \roo$ satisfying the line condition for some $N_\roo\in \N$, then 
for 
$\nu_p$-almost all $x \in E$ the tangent sets are of the form $\OO((\R\times C) \cap B(0,1))$, 
where 
$C$ is a closed porous set and $\OO$ is the rotation taking $(0,1)$ to 
$\overline{\vartheta}_1(\pi^{-1}(x))$. 
\end{proposition}

\begin{proof}
  Fix $\roo>0$. By the line condition, there exist a set $E_\roo \subset E$ with $\nu_p(E_\roo)\geq 
1-\roo$ and an integer $N_\roo\in\N$, where $$\alpha_{2}(\iin{t})  \geq \min\{\alpha_2(\ii) : 
\ii\in 
I^{N_\roo}\} t$$ for all small $t>0$. This means that the option \eqref{tozero} in the proof of 
Lemma \ref{lem:modifiedtangents} is impossible in this set. Hence, a closed porous set is the only 
possible outcome. Letting $\roo\downarrow 0$ proves the claim. 
\end{proof}

\begin{remark}
  Without the line condition, it is certainly possible that the tangent sets are of the form 
$\OO((\ell \times \{ 0 \}) \cap B(0,1))$ for a suitable rotation $\OO$. For example, consider a 
self-affine carpet for which the first level construction rectangles are horizontally aligned and 
disjoint, and their projection onto the $x$-axis is a line segment but, except for the vertical 
edges, no construction rectangle is above another. It is evident that at a generic point the 
tangent 
sets are of the form $(\ell \times \{ 0 \}) \cap B(0,1)$, where $\ell$ is an interval containing at 
least one of the intervals $[-1,0]$ and $[0,1]$.
\end{remark}

With Proposition \ref{cor:nopoint}, we are now ready to prove Corollary \ref{cor:carpet}.

\begin{proof}[Proof of Corollary \ref{cor:carpet}]
  To apply Proposition \ref{cor:nopoint}, we have to verify the assumptions of Theorem \ref{main} 
and check that the line condition holds. The assumption \eqref{as:separation} is trivially 
satisfied 
and, by recalling Remark \ref{re:ontheassumptions2}, the condition \eqref{cor:carpet lyapunov} of 
Corollary \ref{cor:carpet} implies the assumption \eqref{as:lyapunov}. Moreover, Remark 
\ref{re:ontheassumptions} guarantees that it suffices to check the projection condition only onto 
the horizontal direction. But this is guaranteed by the condition \eqref{carpet_line} of Corollary 
\ref{cor:carpet} since every line in the vertical direction meets at least two of the first level 
construction rectangles. This also means that the whole set $E$ satisfies the line condition (with 
$N=1$). Proposition \ref{cor:nopoint} thus applies and shows that the tangent sets at almost all 
points are of the form $(\R \times C) \cap B(0,1)$, where $C$ is a closed porous set. It remains to 
prove that $C$ does not contain any isolated 
points. Here one can argue in exactly the same way as in the proof of \cite[Theorem 
1]{BandtKaenmaki2013}.
\end{proof}

\begin{figure}[!t]
\subfigure[Scale=1]{\includegraphics[width=3.5cm]{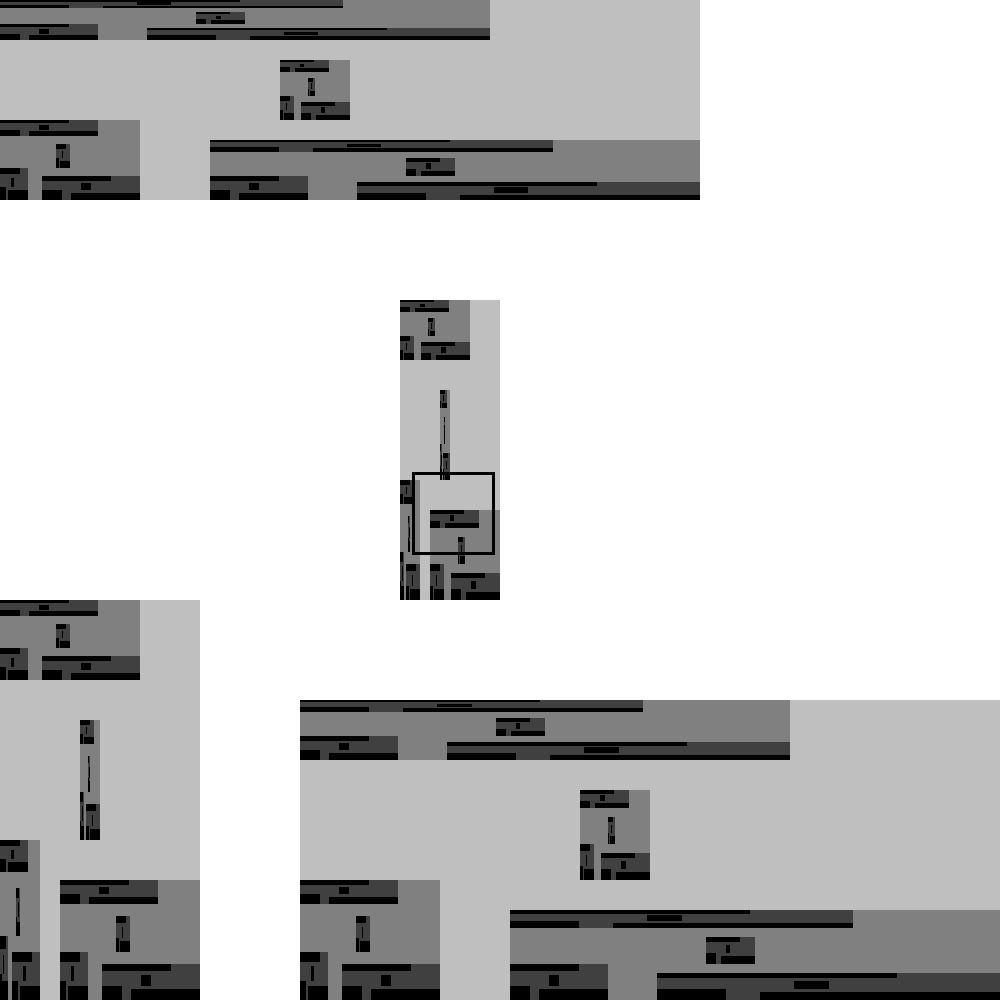}\label{zoom0}}\quad
\subfigure[Scale=0.04]{\includegraphics[width=3.5cm]{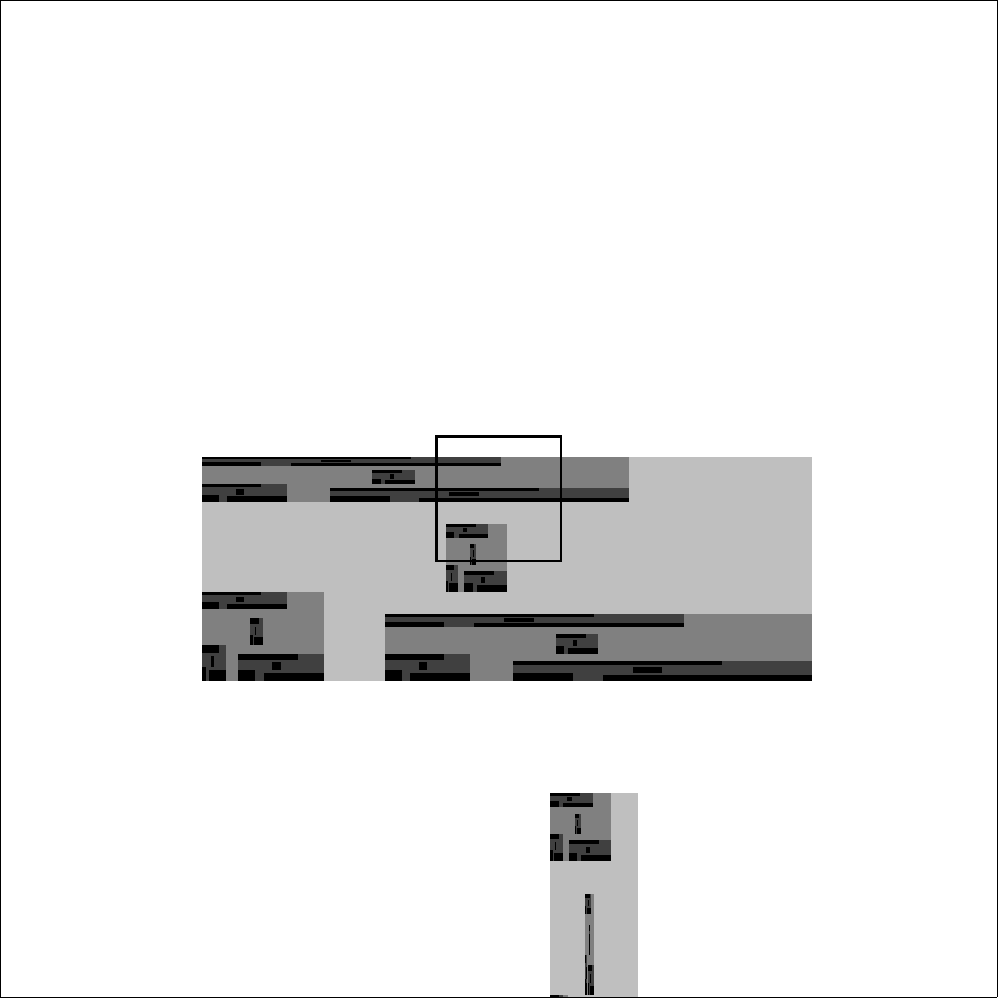}\label{zoom1}}\quad
\subfigure[Scale=0.005]{\includegraphics[width=3.5cm]{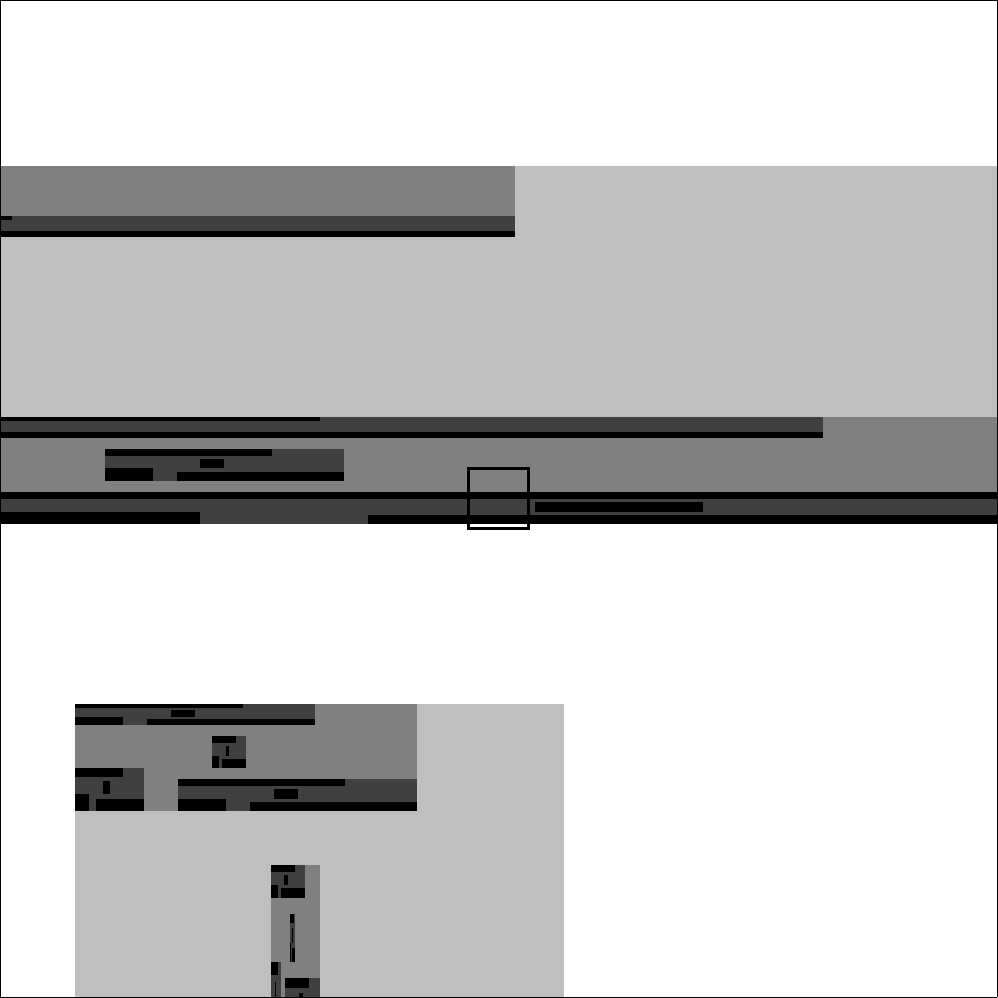}\label{zoom2}}\quad
\subfigure[Scale=0.0003]{\includegraphics[width=3.5cm]{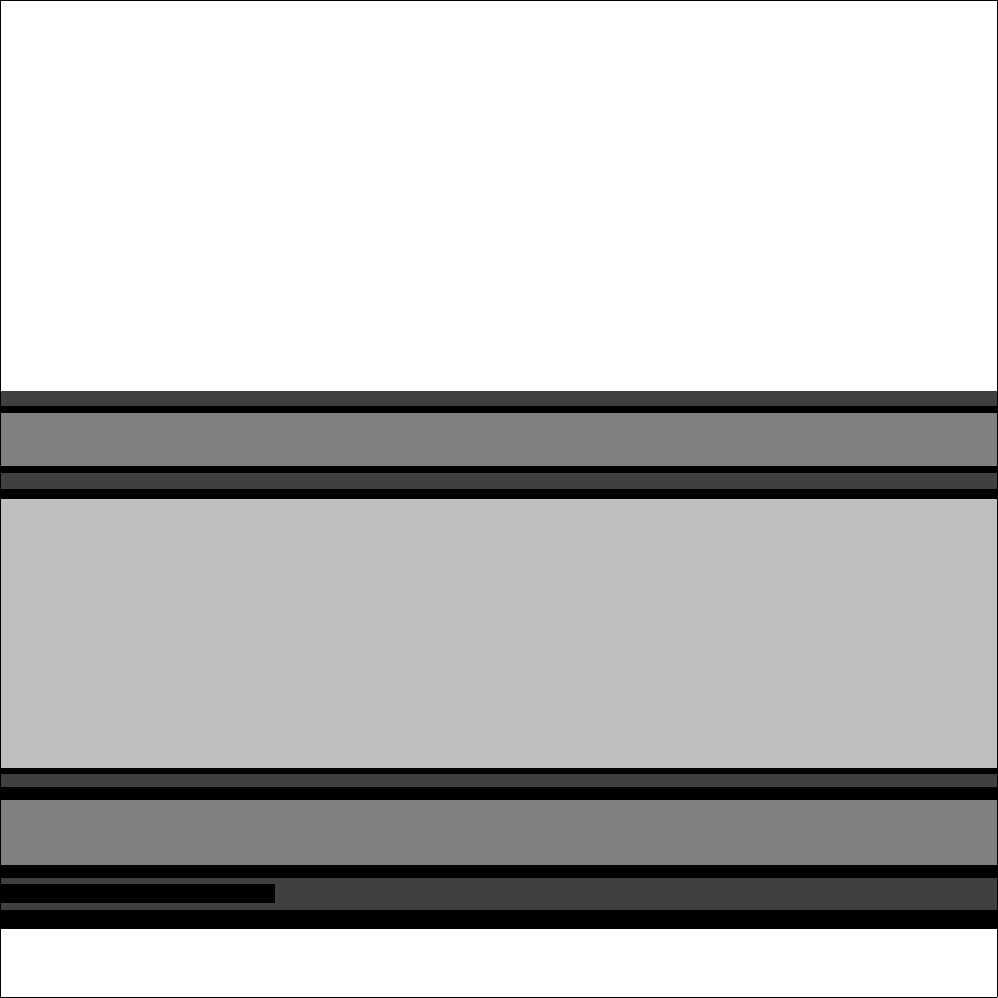}\label{zoom3}}
\caption{The tangent structure arises as one magnifies deeper and deeper in Example 
\ref{IFSexamples}. Here the point where we zoom upon is (0.453846, 0.486659). For illustrative 
purposes we have used a square screen instead of a circular one.}
\label{zoompicture}
\end{figure}

\begin{example}\label{IFSexamples}
  We exhibit an example for which Proposition \ref{cor:nopoint} applies. Let
  \begin{align*}
  f_1(x) &=
  \begin{pmatrix}
    0.2 & 0 \\
    0 & 0.4
  \end{pmatrix} x,
  & f_2(x) &=
  \begin{pmatrix}
    0.7 & 0 \\
    0 & 0.3
  \end{pmatrix} x +
  \begin{pmatrix}
    0.3 \\ 0
  \end{pmatrix}, \\
  f_3(x) &=
  \begin{pmatrix}
    0.7 & 0 \\
    0 & 0.2
  \end{pmatrix} x +
  \begin{pmatrix}
    0 \\ 0.8
  \end{pmatrix},
  & f_4(x) &=
  \begin{pmatrix}
    0.1 & 0 \\
    0 & 0.3
  \end{pmatrix} x +
  \begin{pmatrix}
    0.4 \\ 0.4
  \end{pmatrix},
  \end{align*}
  and $p = (\tfrac16, \tfrac13, \tfrac13, \tfrac16)$. See Figure \ref{zoom0} for an illustration.

  Let us verify the assumptions of Theorem \ref{main} and check that the line condition is 
satisfied. By looking at Figure \ref{zoom0}, it is apparent that the assumption 
\eqref{as:separation} is satisfied. Verifying the assumption \eqref{as:lyapunov} just requires 
checking that \eqref{eq:3.1} holds. For the assumption \eqref{as:projection} to hold, since the 
right-hand side of \eqref{eq:3.1} is greater than the left-hand side in this example, we only need 
to check that the projection onto the $x$-axis is a line-segment. This is again clear from Figure 
\ref{zoom0}. It is also clear from the picture that the set $A_\lambda = ([0,1-\lambda] \times 
[0,1]) \cap E$ satisfies the line condition for all $\lambda>0$. Since the Bernoulli measure 
$\nu_p$ 
has no atoms, we see that for each $\roo>0$ there is $\lambda(\roo)>0$ such that 
$\nu_p(A_{\lambda(\roo)}) \ge 1-\roo$. 
  
  Therefore, according to Proposition \ref{cor:nopoint}, for $\nu_p$-almost every $x \in E$ the 
tangent sets are of the form $(\R \times C) \cap B(0,1)$, where $C$ is a closed porous set. To 
experiment this result in Figure \ref{zoompicture}, observe how the tangent structure begins to 
``converge'' as we magnify deeper and deeper into the set $E$. In Figures \ref{zoom0}--\ref{zoom2}, 
one can see the screen of the next magnification. Each picture shows four subsequent levels of the 
construction so that the level $n(x,t)$ is white and the color gets darker as $n$ increases.
\end{example}

The next example shows that Theorem \ref{main} applies also outside of the class of self-affine 
carpets.

\begin{example}\label{ex:noncarpet}
 Let $R_1$ be the rotation of angle $\pi/2$ (counter-clockwise) and set
  \begin{align*}
  f_1(x) =
  \begin{pmatrix}
    0.5 & 0 \\
    0 & 0.05 \\
  \end{pmatrix}
  x +
  \begin{pmatrix}
    0.4 \\
    -0.4
  \end{pmatrix},\quad
  f_5(x) =
  \begin{pmatrix}
  0.2 & 0 \\
  0 & 0.05 \\
  \end{pmatrix}
  x +
  \begin{pmatrix}
    0.8 \\
    0
  \end{pmatrix},
  \end{align*}
  and $f_{k+i}(x)=R_1^i f_k(x)$ for $k\in\{1,5\}$ and $i\in\{1,2,3\}$. Finally, let $f_9(x)=0.2 
R_2(x)$, where $R_2$ is a rotation of angle $\theta$ so that $\pi/\theta$ is irrational.
  
  From Figure \ref{Example1} one sees that the strong separation condition and the condition given 
in Remark \ref{re:ontheassumptions}, to ensure the assumption \eqref{as:projection} of Theorem 
\ref{main}, are satisfied. When considering the projection condition, it is helpful to notice that 
the images of $[-1,1] \times \{ 0 \}$ and $\{ 0 \} \times [-1,1]$ under $f_i$ are contained in 
$f_i(X)$, where $X$ is the convex hull of $E$. Examining how these images are located in 
$f_i(B(0,1))$ helps to verify the condition given in Remark \ref{re:ontheassumptions}. In this 
reasoning, it is essential that $f_5$, $f_6$, $f_7$, and $f_8$ fix points in the unit circle. 
Observe that it is not necessary for the linear parts of $f_1, \ldots, f_8$ to be diagonal.
  
  To verify the assumption \eqref{as:lyapunov} of Theorem \ref{main}, we rely on Remark 
\ref{re:ontheassumptions2} and check that the iterated function system is pinching and twisting. 
Given a constant $C>1$, consider the word $\ii$ with $\ii(n)=1$ for all $n\in\N$. Now 
$\alpha_1(\iin{n})=0.5^n$ and $\alpha_2(\iin{n})=0.05^n$, so it is clear that for some large $n$ we 
have that $\alpha_1(\iin{n})>C\alpha_2(\iin{n})$. On the other hand, given a finite set of vectors 
$\{v,w_1,\ldots,w_n\}$, consider the word $\ii$ with $\ii(n)=9$ for all $n\in\N$. Since 
$\pi/\theta$ 
is irrational, the set $\{f_{\iii|_n}(v)/\|f_{\iii|_n}(v)\|\}_{n=1}^\infty$ is dense in $S^1$, and 
so there exists $n$ so that $f_{\iin{n}}(v)$ is not parallel to any $w\in\{w_1,\ldots,w_n\}$. Thus 
the Lyapunov exponents are different for any Bernoulli measure.
\end{example}

\begin{figure}[!t]
  \subfigure[]{\includegraphics[width=5cm]{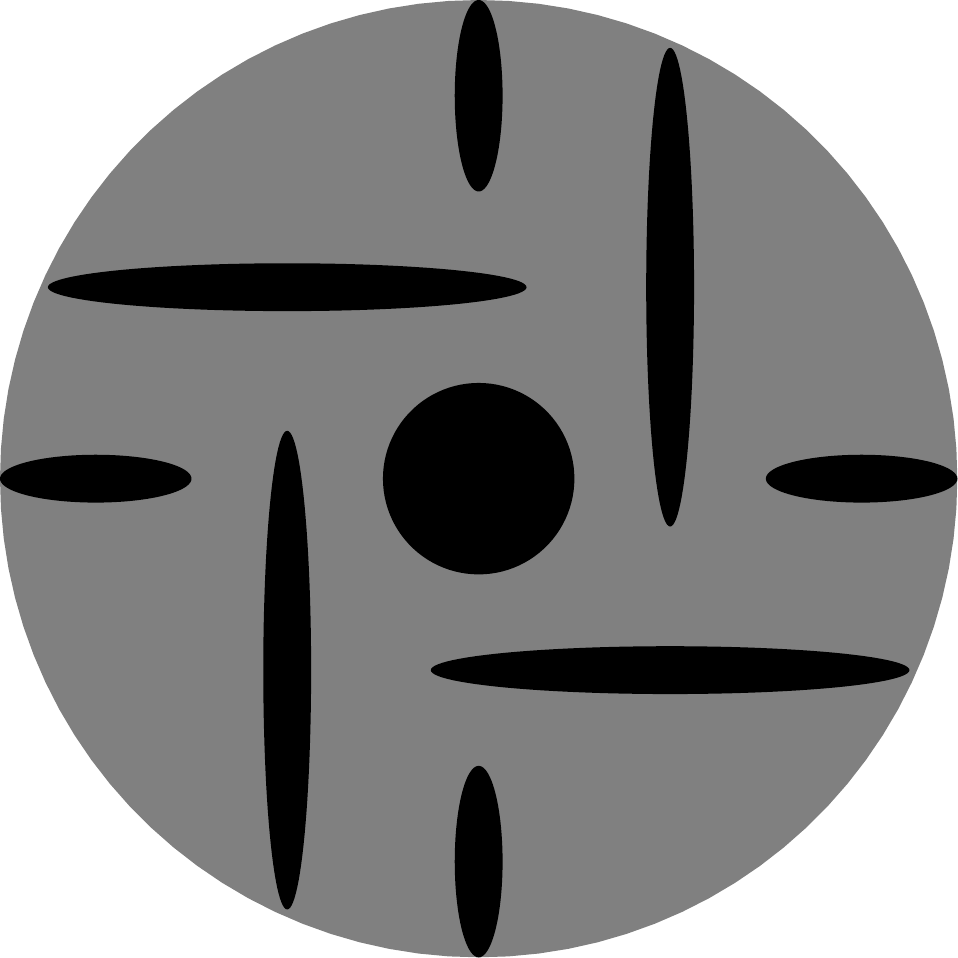}\label{Example1}} 
  \subfigure[]{\includegraphics[width=5cm]{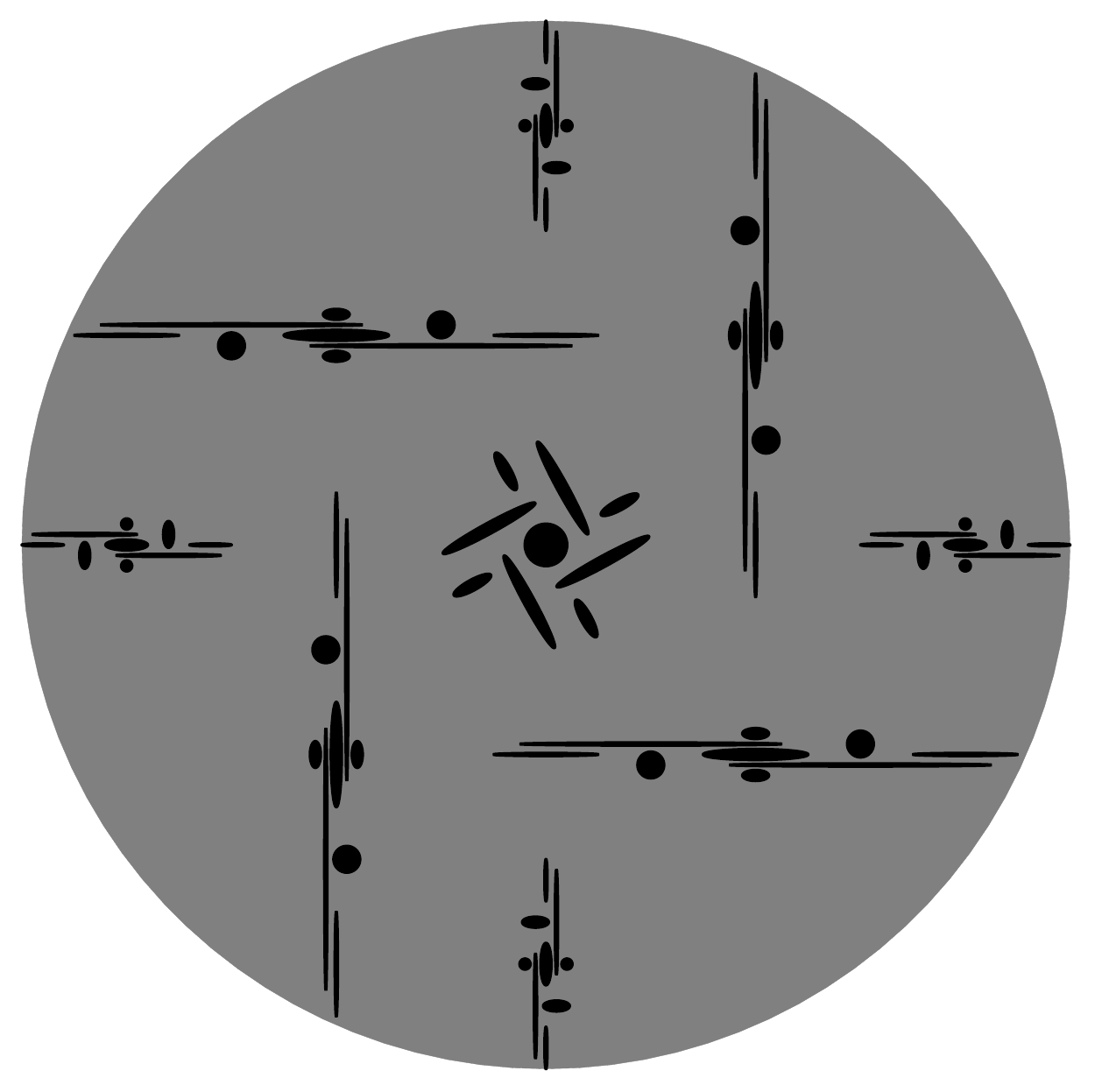}\label{Example2}}
  \caption{The pictures depict the first and second level images of $B(0,1)$ inside $B(0,1)$ for 
the 
iterated function system of Example \ref{ex:noncarpet}.}
  \label{genex}
\end{figure}

We finish the discussion by giving an example which shows that the assumption \eqref{as:projection} 
in Theorem \ref{main} is necessary.

\begin{example}\label{ex:noprojection}
  Consider the set $F=C_{1/3}\times C_{1/4}$, where $C_\lambda$ is the Cantor set constructed by 
using the contraction ratio $\lambda$. Then $F$ satisfies the assumptions \eqref{as:separation} and 
\eqref{as:lyapunov} of Theorem \ref{main}, with $p=(\tfrac14,\tfrac14,\tfrac14,\tfrac14)$.

  Let $T$ be the tangent set of $F$ along a sequence $(t_n)$ at a point  $(x^{(1)}, x^{(2)})\in F$. 
Denote by $T_1$ and $T_2$ the tangent sets of $C_{1/3}$ and $C_{1/4}$, at points $x^{(1)}$ and 
$x^{(2)}$ along the sequence $(t_n)$. (If necessary, take the sub-sequence of $(t_n)$ along which 
these both exist.) Fix $\epsilon>0$. Denote by $A_\epsilon$ the set of points within distance 
$\epsilon$ from a set $A$. For notational simplicity, we interpret $\ZZ_{x,t}$ to be a mapping 
defined on a square centered at $x$ with side length $t$. We use the same symbol $\ZZ_{x,t}$ to 
denote the corresponding zoom on the real line. Then, for large enough $n$, 
  \begin{align*}
    \ZZ_{x, t_n}(Q(x, t_n)\cap F)&=\ZZ_{x, t_n}(Q(x^{(1)}, t_n)\cap C_{1/3})\times \ZZ_{x, 
t_n}(Q(x^{(2)}, t_n)\cap C_{1/4})\\
    &\subset (T_1)_\epsilon\times (T_2)_\epsilon\subset (T_1\times T_2)_{3\epsilon}, 
  \end{align*}
  where the first inclusion uses the fact that $F$ is a product set, and the second the definition 
of tangent sets. Similarly, for large $n$, 
  \[
  \ZZ_{x, t_n}(Q(x, t_n)\cap F)_{3\epsilon}\supset T_1\times T_2. 
  \] 
  This proves that a tangent set at a point of $F$ is a product of tangent sets of $C_{1/3}$ and 
$C_{1/4}$.  

  By \cite{BedfordFisher1996}, tangent sets of a Cantor set are $C^{1+\gamma}$-images of the set 
itself. This implies that the tangent sets at points of $F$ are of the form $T_1\times T_2$ where 
$T_1$ and $T_2$ are $C^{1 + \gamma}$-images of the Cantor sets. In particular, the claim of Theorem 
\ref{main} does not hold. 
\end{example}

\bibliographystyle{abbrv}
\bibliography{Bibliography}

\end{document}